\documentclass{custom} 

\title[Optimal dimension dependence of MALA]{Optimal dimension dependence of the Metropolis-Adjusted Langevin Algorithm}
\usepackage{bbm}
\usepackage{enumitem}
\usepackage{graphicx}
\usepackage{mathtools}
\usepackage{microtype}
\usepackage{tikz}
\usepackage{times}
\usepackage[Symbolsmallscale]{upgreek}
\usepackage{thm-restate} 
\usepackage{style}

\makeatletter
\def\set@curr@file#1{\def\@curr@file{#1}} 
\makeatother
\usepackage[load-configurations=version-1]{siunitx} 

\coltauthor{\Name{Sinho Chewi} \Email{schewi@mit.edu}\\
 \Name{Chen Lu} \Email{chenl819@mit.edu}\\
 \Name{Kwangjun Ahn} \Email{kjahn@mit.edu}\\
 \Name{Xiang Cheng} \Email{chengx@mit.edu}\\
 \Name{Thibaut {Le Gouic}} \Email{tlegouic@mit.edu}\\
 \Name{Philippe Rigollet} \Email{rigollet@mit.edu}\\
 \addr MIT}

\begin{document}

\maketitle

\begin{abstract}%
    
    
    Conventional wisdom in the sampling literature, backed by a popular diffusion scaling limit, suggests that the mixing time of the Metropolis-Adjusted Langevin Algorithm (MALA) scales as $O(d^{1/3})$, where $d$ is the dimension. However, the diffusion scaling limit requires stringent assumptions on the target distribution and is asymptotic in nature. In contrast, the best known non-asymptotic mixing time bound for MALA on the class of log-smooth and strongly log-concave distributions is $O(d)$. In this work, we establish that the mixing time of MALA on this class of target distributions is $\widetilde\Theta(d^{1/2})$ under a warm start. 
    Our upper bound proof introduces a new technique based on a projection characterization of the Metropolis adjustment which reduces the study of MALA to the well-studied discretization analysis of the Langevin SDE and bypasses direct computation of the acceptance probability.
\end{abstract}


\section{Introduction}


Sampling from a target distribution is a central problem that arises in many areas of scientific computing and statistics~\citep{liu2008monte, robert2013monte}. The class of Metropolis-Hastings (MH) adjusted algorithms~\citep{metropolis1953equation, hastings1970monte}, which includes the Random Walk Metropolis algorithm (RWM), the Metropolis-Adjusted Langevin Algorithm (MALA), and Hamiltonian Monte Carlo (HMC), is particularly popular in practice. As such, their convergence properties are of central theoretical and practical interest. More specifically, with the ever-growing size of sample spaces, a precise characterization of how dimension affects convergence rates is a necessary step to develop a better understanding and, ultimately, practical guidelines for this suite of algorithms. In this work, we address this pressing question by characterizing the dimension dependence of MALA over a natural class of distributions. 


Formally, we consider the task of sampling from a target distribution $\pi$ supported on $\R^d$, with density $\pi(\bx) \propto \exp(-V(\bx))$, where $V:\R^d\to \R$ is a strongly convex and smooth potential.  \cite{roberts1997weak} initiated the study of dimension dependence of RWM by means of an asymptotic framework: namely, when $\pi$ is a product distribution, a scaling limit exists for RWM as the dimension tends to infinity with a dimension-dependent step size $h \approx d^{-1}$, thereby suggesting that the number of steps needed for RWM to reach stationarity is on the order of $d$. Subsequently,~\cite{roberts1998optimal}  \citep[see also][]{pillaistuartthiery2012optimalscaling} extended the scaling limit approach to MALA, suggesting that the dimension dependence for MALA is $d^{1/3}$ for sufficiently regular potentials and step size $h \approx d^{-1/3}$. Beyond its theoretical implications, this result has had a tremendous practical impact by guiding the choice of step size for MALA even for distributions far beyond the scope of their seminal paper. Understanding the applicability of this result, and ultimately the optimal rate of convergence of MALA, requires a careful inspection of the framework laid out in~\cite{roberts1998optimal}. It turns out that it is rather limited in several aspects. Perhaps most notably, it requires $\pi$ to be a product distribution, which excludes distributions with complex dependence structures that are now routinely encountered in high-dimensional statistics. Moreover, it applies only to potentials $V$ with higher-order derivatives; this is not a mere technical artefact since the limit acceptance probability of MALA as $d \to \infty$ involves the third derivative of $V$. Finally, the asymptotic nature of the scaling limit result only suggests dimension dependence in the asymptotic limit as $d\to \infty$, so it potentially washes away important effects that may arise for finite $d$. 

Thus it is natural to investigate the rate of convergence of MALA from a perspective that is now customary in the machine learning and optimization literature: by establishing non-asymptotic rates of convergence that hold uniformly over natural classes of target distributions which go beyond product distributions. 
We begin with the simplest and most natural setting and ask:

\begin{quote}
    What is the optimal dimension dependence of the mixing time of MALA uniformly over the class of $\alpha$-strongly convex and $\beta$-smooth potentials?
\end{quote}

Interestingly, and somewhat surprisingly, we show that while the rate $d^{1/3}$ originally established by~\cite{roberts1998optimal} is indeed optimal for some product distributions such as the standard Gaussian, it is not optimal uniformly over the class of smooth and strongly convex potentials of interest in this work. In fact, for any choice of $d$, we exhibit a product distribution with infinitely differentiable potential on which MALA requires a stepsize much smaller than $d^{-1/3}$, thus resulting in a worse mixing time. This construction confirms the limitations of the scaling limit approach to establishing optimal dimension dependence.

\medskip{}

\noindent{\bf Related work.}
The non-asymptotic performance of sampling algorithms uniformly over the class of smooth and strongly convex potentials has been the object of intense research activity recently. For example, \cite{dwivedi2019log,chenetal2020hmc} show that on this class of potentials, RWM can draw samples with at most $\varepsilon$ error in chi-squared divergence with $O(d\log \frac 1 \varepsilon)$ steps, thereby providing a non-asymptotic affirmation of the scaling limit of~\cite{roberts1997weak}. However, far less is known about optimal rates for MALA. The current best result for MALA on the class of smooth and strongly convex potentials is the paper~\cite{chenetal2020hmc}, which proves a complexity of $O(d\log \frac 1 \varepsilon)$ steps to achieve $\varepsilon$ error in chi-squared divergence. They also raise the question of whether there is a gap between the complexities of RWH and MALA\@.


\cite{mangoubi2019nonconvex} took a direct aim at improving the dimension dependence of mixing time bounds for MALA. They succeeded in obtaining a bound of $O(d^{2/3})$ albeit at the cost of stringent hypotheses. More specifically, they assume bounds on the third and fourth derivatives of the potential $V$; when these bounds are $O(1)$ (which is true for the standard Gaussian) then their mixing time is $O(d^{2/3})$; see the discussion in~\cite{chenetal2020hmc}.

\medskip{}
\noindent{\bf Our contributions.} In this work, we show that the mixing time in chi-squared divergence for MALA on the class of smooth and strongly convex potentials with a warm start is $\widetilde \Theta(d^{1/2})$. Our result consists of two parts: an upper bound on the mixing time which improves to optimality prior results such as~\cite{dwivedi2019log, chenetal2020hmc}, as well as the construction of smooth and strongly convex potentials on which the mixing time of MALA is no better than $d^{1/2}$.

In addition to establishing the optimal dimension dependence for MALA, our result is also one of the strongest guarantees for sampling with a warm start to-date, irrespective of the algorithm. Indeed, the algorithms which achieve similar or better dimension dependence compared to our result are: the underdamped Langevin algorithm \citep[][ $O(d^{1/2})$]{chengetal2018underdamped}, the higher-order Langevin algorithm \citep[][$O(d^{1/2})$]{mou2020highorder}, the randomized midpoint discretization of underdamped Langevin \citep[][$O(d^{1/3})$]{shen2019randomized}, and Hamiltonian Monte Carlo \citep[][$O(d^{1/4})$]{mangoubivishnoi2018hmc}.
However, the dependence of these results on $1/\varepsilon$ is polynomial, whereas our dependence on $1/\varepsilon$ is polylogarithmic. Therefore, for a wide range of accuracy values which are inverse polynomial in the dimension (e.g., $\varepsilon = 1/d$), our result attains the best-known dependence on the dimension.

In order to prove our upper bound on the mixing time, we introduce new techniques based on the  characterization of the Metropolis filter as a projection of the Markov transition kernel in expected $L_1$ distance~\citep{billeradiaconis2001mhprojection}. Our techniques effectively reduce the problem of bounding the mixing time to controlling the discretization error between the continuous-time and discretized Langevin processes, which has been extensively studied in the sampling literature. We do not aim to give a comprehensive bibliography here, but we note that our discretization analysis is closest to the papers~\cite{dalalyan2012sparse, dalalyan2014theoretical}. In this way, our upper bound has the potential to connect the vast literature on discretization of SDEs with the more difficult analysis of Metropolised algorithms, although it is likely that further innovations are necessary before the study of the latter is completely reduced to the former.

\medskip{}
\noindent{\bf Notation.} We use the symbol $\bx$ to denote a $d$-dimensional vector, and the plain symbol $x$ to denote a scalar variable. We abuse notation by identifying measures with their densities (w.r.t.\ Lebesgue measure); thus, for instance, $\pi$ represents the stationary distribution (a measure), and the notation $\pi(\bx)$ refers to the corresponding density evaluated at $\bx$.

\section{Preliminaries}

\subsection{Assumptions}
\label{sec:assumptions}

We consider the problem of sampling from a distribution $\pi$ supported on $\R^d$. The density of the distribution is given by $\pi(\bx) \propto \exp(-V(\bx))$, and we refer to $V:\R^d \to \R$ as the \emph{potential}. Throughout the paper, we will assume that $V$ is twice continuously differentiable, $\alpha$-strongly convex, and $\beta$-smooth, meaning
\begin{align*}
    \alpha I_d \preceq \nabla^2 V(\bx) \preceq \beta I_d, \qquad \forall\, \bx \in \R^d.
\end{align*}
We assume that $\beta \ge 1 \ge \alpha$, and we denote by $\kappa := \beta/\alpha$ the \emph{condition number}.

For the sake of normalization, we assume that $V(\bs 0) = \min V = 0$, so that $\nabla V(\bs 0) = \bs 0$.

\subsection{Metropolis-Adjusted Langevin Algorithm (MALA)}
Before stating our main results, we give some background on MALA and tools for establishing convergence rates of Markov chains. 

Given a step size $h>0$, MALA produces a sequence ${(\bx_n)}_{n \ge 0}$ of random points in $\R^d$ as follows. First, MALA is initialized at $\bx_0\sim \mu_0$. Then, for $n \ge 0$, repeat the following two-step procedure:
   \begin{enumerate}
   \item Proposal step: sample $\by_{n+1}\sim Q(\bx_n, \cdot)$, where
   \begin{align*}
    Q(\bx,\cdot)
    &:= \frac{1}{{(4\uppi h)}^{d/2}} \exp\Bigl( - \frac{\norm{\,\cdot -\bx + h\nabla V(\bx)}^2}{4h} \Bigr).
\end{align*}
This proposal density corresponds to one step of the unadjusted Langevin algorithm.
\item Accept-reject step: set 
\[
\bx_{n+1} = \left\{
\begin{array}{ll}
\by_{n+1} & \text{with probability $A(\bx_n, \by_{n+1})$} \\
\bx_{n} & \text{with probability $1-A(\bx_n, \by_{n+1})$} \\
\end{array}
\right.
\]
where the acceptance probability is given by
\begin{align}\label{eq:accept_prob}
    A(\bx,\by)
    &:= 1 \wedge a(\bx, \by) \ , \qquad a(\bx, \by) := \frac{\pi(\by) Q(\by, \bx)}{\pi(\bx) Q(\bx,\by)}\,.
\end{align}
   \end{enumerate}

It is well-known that MALA outputs a sequence of random variables ${(\bx_n)}_{n \ge 0}$ that forms a reversible Markov chain with stationary distribution $\pi$ and Markov transition kernel given by
\begin{align}
\begin{aligned}\label{eq:metro_adjust}
    T(\bx, \by)&= [1-a(\bx)] \, \delta_{\bx}(\by) + Q(\bx, \by) A(\bx,\by),\\
    A(\bx) &=  \int Q(\bx, \by) A(\bx, \by) \, \D \by \geq 0.
\end{aligned}
\end{align}
For the rest of the paper, it is important to note that $A$, $Q$, etc.\ depend on the step size $h$.

There are many choices to measure proximity of the MALA output with the target distribution. In this work, we focus on the Total Variation distance ($\TV$), the Kullback-Leibler divergence ($\KL$), the chi-squared divergence ($\chi^2$), and the 2-Wasserstein distance ($W_2$). Given a measure of discrepancy  $\msf d$ between probability measures, 
we define the mixing time, with initial distribution $\mu_0$, as follows:
\begin{align*}
    \tau_{\rm mix}(\varepsilon, \mu_0; \msf d) := \inf\{ n \in \N \,:\, \mc \bx_0\sim \mu_0 , \ \msf d(\mu_n, \pi) \leq \varepsilon\} \,.
\end{align*}

Extensions to other discrepancies, such as the $p$-Wasserstein distance for $p \le 2$ or the Hellinger distance, are straightforward and omitted for brevity.

The mixing time of a Markov chain is governed by its spectral gap, which we now introduce.
To that end, recall that the \emph{Dirichlet form} associated with the MALA kernel $T$ is the quadratic form
\begin{align*}
    \mc E (f, g) = \E_\pi[f \,({\id} -T)g], \qquad f, g \in L^2(\pi),
\end{align*}
where $(Tg)(\bx) := \int g(\by)\, T(\bx, \D \by)$. The \emph{spectral gap} is defined as
\begin{align}\label{eq:lambda}\tag{$\lambda$}
    \lambda
    := \inf\Bigl\{\frac{\mc E(f, f)}{\Var f}\,:\, f\in L^2(\pi), \; \var f > 0\Bigr\} \,.
\end{align}
Since it is often difficult to control the spectral gap directly, it is also convenient to introduce the \emph{conductance}, defined as
\begin{align}\label{eq:C}\tag{$\msf C$}
    \msf C := \inf\Bigl\{\frac{\int_S T(\bx, S^\comp) \, \pi(\D \bx)}{\pi(S)} \,:\, S \subseteq \R^d, \; \pi(S) \le \frac{1}{2} \Bigr\} \,.
\end{align}
By Cheeger's inequality~\citep{lawlersokal1998cheeger}, it holds that
\begin{align}\label{eq:cheeger}
    \msf C^2
    &\lesssim \lambda
    \lesssim \msf C.
\end{align}

\section{The Gaussian case}

As our work is motivated by the diffusion scaling limit of~\cite{roberts1998optimal}, which predicts a $d^{1/3}$ mixing time for MALA, it is natural to begin our investigations by asking whether this is indeed the correct order of the mixing time in the simplest possible setting: namely, when $\pi$ is the standard Gaussian distribution. Our first contribution is to establish that it is indeed the case even for finite $d$. We formulate here an informal result and postpone a more detailed statement together with a proof to Appendix~\ref{scn:gaussian}. Though it is expected, this result appears to be new.

\begin{theorem}[informal]
    If the target distribution $\pi$ is the standard Gaussian distribution, then the mixing time of MALA under a warm start is $\Theta(d^{1/3})$, and is achieved with step size $h \approx d^{-1/3}$.
\end{theorem}


    The proof of this result is based on explicit calculations. While limited to the Gaussian case, its inspection is instructive for potential extensions to other distributions.
    
    On the one hand, the upper bound on the mixing time relies on fine cancellations in the acceptance probability using the explicit form of the Gaussian distribution, which is unavailable for more general potentials. In general, it is difficult to control the acceptance probability directly, and this seems to be the main obstacle to sharpening the mixing time bound in~\cite{dwivedi2019log}. This observation motivates us to seek an indirect way of controlling the acceptance probability in the next section.
    
    On the other hand, while the Gaussian target distribution readily yields a lower bound over the class of potentials with smooth and strongly convex potentials, it turns out to be too loose to address the optimality of MALA. In Section~\ref{scn:lower_bd}, we show that a tighter lower bound may be achieved using a carefully chosen perturbation of the Gaussian distribution.

\section{Upper bound}\label{scn:upper_bd}

In order to prove an upper bound on the mixing time of MALA, we assume that we have access to a \emph{warm start}. This is a common assumption which has been employed in previous works on MALA, e.g.~\cite{dwivedi2019log, mangoubi2019nonconvex, chenetal2020hmc}.

\begin{definition}[warm start]
\label{def:warmstart}
We say that the initial distribution $\mu_0$ is \emph{$M_0$-warm} with respect to $\pi$ if for any Borel set $E \subseteq \R^d$, it holds that $\mu_0(E) \le M_0 \pi(E)$. When clear from the context, we simply say that an algorithm has a $M_0$-warm start to indicate that it is initialized at an  \emph{$M_0$-warm} distribution and omit reference to the target distribution.
\end{definition}

We now state our upper bound on the mixing time of MALA, which shows that under a warm start the mixing time of MALA is $\widetilde O(\sqrt d)$.

\begin{theorem}\label{thm:mala_upper_bd}
    Fix $\varepsilon > 0$ and consider a target distribution $\pi$ satisfying the assumptions of Section~\ref{sec:assumptions}. Then MALA with a $M_0$-warm start and step size 
    \begin{align*}
        h
        &= \frac{c \alpha^{1/2}}{\beta^{4/3} d^{1/2} \log(d\kappa M_0/\varepsilon)}
    \end{align*}
for a sufficiently small absolute constant $c > 0$,    has mixing time given by
    \begin{align*}
        \tau_{\rm mix}(\varepsilon, \mu_0; \msf d)
        &\lesssim \frac{\beta^{4/3} d^{1/2}}{\alpha^{3/2}} \log\Bigl(\frac{M_0}{\varepsilon}\Bigr) \log\Bigl( d\kappa + \frac{M_0}{\varepsilon}\Bigr)\,.
    \end{align*}
    for each of the distances
    \begin{align*}
        \msf d \in \{\TV, \; \sqrt{\KL}, \; \sqrt{\chi^2}, \; \sqrt\alpha \,W_2\}\,.
    \end{align*}
\end{theorem}

    The main properties of strongly log-concave distributions that we use in the proof are summarized in Lemma~\ref{lem:properties_of_strong_log_concave}. As long as $\pi$ satisfies these properties, the upper bound technique may be applied under weaker assumptions, e.g., a log-Sobolev inequality.
    We do not pursue these extensions further in this paper.
    
    We primarily work with the total variation distance to establish the above upper bound on the mixing time and translate this result to the chi-squared divergence by leveraging $M_0$-warmness of all the iterates of the MALA chain. In turn, this result extends to the KL divergence using a standard comparison inequality \citep[see, e.g.,][Chapter~2]{tsybakov2009nonparametric} and ultimately to the Wasserstein distance using Talagrand’s transportation inequality for strongly log-concave distributions.

    The bound above is likely not sharp in terms of the accuracy parameter $\varepsilon$ and the warm start parameter $M_0$. Indeed, we expect the dependency on the accuracy parameter to be $\log(1/\varepsilon)$, and the paper~\cite{chenetal2020hmc} develops a method, based on the conductance profile, to reduce the warm start dependence to $\log \log M_0$. Since the quantity $\log M_0$ can introduce additional dimensional factors under a feasible start~\citep{dwivedi2019log}, it is important to improve the dependency on $M_0$. We leave open the question of refining our techniques to achieve these improvements.


Since our upper bound proof may be of interest for analyzing other sampling algorithms based on Metropolis-Hastings filters, we now proceed to give a technical overview of the ideas involved in the upper bound. Throughout, we use the notation $Q_{\bx}(\cdot)$, $T_{\bx}(\cdot)$, etc.\ as a shorthand for the kernels $Q(\bx,\cdot)$, $T(\bx, \cdot)$, etc.

We begin by describing the approach of~\cite{dwivedi2019log}, which will serve as a reference. The standard technique for bounding the conductance of geometric random walks is the following lemma~\citep[see, e.g.,][Lemma 13]{lee2018convergence}.

\begin{lemma}\label{lem:standard_conductance_lem}
    Suppose that for all $\bx, \by \in \R^d$ with $\norm{\bx-\by} \le r$, it holds that $\norm{T_{\bx} - T_{\by}}_{\rm TV} \le 3/4$.
    Then, the conductance of the MALA chain satisfies $\msf C \gtrsim \sqrt{\alpha} r$.
\end{lemma}
In light of this lemma,~\cite{dwivedi2019log} considers the following decomposition:
\begin{align}\label{eq:TV_decomposition}
    \norm{T_{\bx} - T_{\by}}_{\rm TV}
    &\le \norm{T_{\bx} - Q_{\bx}}_{\rm TV} + \norm{Q_{\bx} - Q_{\by}}_{\rm TV} + \norm{T_{\by} - Q_{\by}}_{\rm TV}.
\end{align}
The middle term is the TV distance between two Gaussian distributions, and using Pinsker's inequality it is straightforward to show that
\begin{align*}
    \norm{Q_{\bx} - Q_{\by}}_{\rm TV}
    &\le \frac{\norm{\bx - \by}}{\sqrt{2h}}\,, \qquad\text{provided}~h \le \frac{2}{\beta}\,,
\end{align*}
see~\cite[Lemma 3]{dwivedi2019log}. On the other hand, bounding the first and third terms in the decomposition~\eqref{eq:TV_decomposition} requires carefully controlling the acceptance probability of MALA\@.
\cite{dwivedi2019log} show that these terms can be controlled when the step size is of order $h \approx 1/d$. An application of Lemma~\ref{lem:standard_conductance_lem} with $r \approx \sqrt h$ yields a conductance bound of $\msf C = \Omega(1/\sqrt d)$ and in turn, a spectral gap bound of $\lambda = \Omega(1/d)$  by Cheeger's inequality~\eqref{eq:cheeger}. Overall, this approach yields a mixing time bound is $O(d)$.

In order to prove a stronger mixing time bound of $\widetilde O(\sqrt d)$, we must consider much larger step sizes (of order $h \approx 1/\sqrt d$), and in this regime, controlling the acceptance probabilities by hand requires a daunting computational effort. In fact,  \cite{roberts1998optimal} already resort to a computer-aided proof to study the asymptotics of the acceptance probability. Our first main idea is to use the well-known fact~\citep{billeradiaconis2001mhprojection} that for any proposal $Q$, the corresponding Metropolis-adjusted kernel $T$ is the closest Markov kernel to $Q$, among all reversible Markov kernels with stationary distribution $\pi$.

\begin{lemma} \label{lem:tv_projection}
Let $Q$ be an atomless proposal kernel, and let $T$ be the kernel obtained from $Q$ by Metropolis adjustment (defined by~\eqref{eq:accept_prob} and~\eqref{eq:metro_adjust}). Let $\bar Q$ be any kernel that is reversible with respect to $\pi$ and has no atoms. Then, for $\bx \sim \pi$, it holds that
\begin{align*}
    \E \norm{T_{\bx} - Q_{\bx}}_{\rm TV}  
    &\le 2\E \norm{\bar Q_{\bx} - Q_{\bx}}_{\rm TV}\, .
    \end{align*}
\end{lemma}
\begin{proof}
    See Appendix~\ref{scn:pointwise_projection_proof}.
\end{proof}
We apply this result by comparing the MALA kernel $T$ with the transition kernel $\bar Q$ of the continuous-time Langevin diffusion run for time $h$.
In other words, $\bar Q(\bx, \cdot)$ is the law of $\bar{\bX}_h$, where ${(\bar{\bX}_t)}_{t\ge 0}$ evolves according to the stochastic differential equation
\begin{align}\label{eq:langevin_sde}
    \D \bar{\bX}_t
    &= -\nabla V(\bar{\bX}_t) \, \D t + \sqrt 2 \, \D \bB_t, \qquad \bar{\bX}_0 = \bx,
\end{align}
and ${(\bB_t)}_{t\ge 0}$ is a standard Brownian motion. Using standard arguments from stochastic calculus (see~\eqref{eq:expectation_tv}), we show that  $\E \norm{\bar Q_{\bx} - Q_{\bx}}_{\rm TV}=O(h\sqrt d)$ (see~\eqref{eq:expectation_tv}). This suggests that we can take the step size to be $h \asymp 1/\sqrt d$. However, since the lemma only controls the first and third terms of the decomposition~\eqref{eq:TV_decomposition} in expectation, it is not enough to yield a good lower bound on the conductance via Lemma~\ref{lem:standard_conductance_lem}. To remedy this, we prove a new pointwise version of the projection characterization of Metropolis adjustment.

\begin{theorem}\label{thm:pointwise_projection}
    Let $Q$ be an atomless proposal kernel, and let $T$ be the kernel obtained from $Q$ by Metropolis adjustment (defined by~\eqref{eq:accept_prob} and~\eqref{eq:metro_adjust}). Let $\bar Q$ be any kernel that is reversible with respect to $\pi$ and has no atoms. Then, for every $\bx \in \R^d$,
    \begin{align}
    \begin{aligned}
        \norm{T_{\bx} - Q_{\bx}}_{\rm TV}
        &\le 2\,\norm{\bar Q_{\bx} - Q_{\bx}}_{\rm TV} + \int \frac{\pi(\by) \bar Q(\by, \bx)}{\pi(\bx)} \, \bigl\lvert \frac{Q(\by,\bx)}{\vphantom{\big|}\bar Q(\by,\bx)} - 1\bigr\rvert \, \D \by.
    \end{aligned}\label{eq:pointwise_TV}
    \end{align}
    Consequently, for any convex increasing function $\Phi : \R_+ \to \R_+$ and $\bx \sim \pi$, $\by \sim \bar Q(\bx, \cdot)$,
    \begin{align}\label{eq:cvx_fn_projection}
        \E\Phi(\norm{T_{\bx} - Q_{\bx}}_{\rm TV})
        &\le \frac{1}{2} \E \Phi(4 \, \norm{\bar Q_{\bx} - Q_{\bx}}_{\rm TV}) + \frac{1}{2} \E \Phi\bigl( 2\,\bigl\lvert \frac{Q(\bx,\by)}{\vphantom{\big|}\bar Q(\bx,\by)} - 1\bigr\rvert\bigr).
    \end{align}
\end{theorem}
\begin{proof}
    See Appendix~\ref{scn:pointwise_projection_proof}.
\end{proof}

\begin{remark}
    If we take the expectation of~\eqref{eq:pointwise_TV} when $\bx \sim \pi$, we obtain
    \begin{align*}
        \E \norm{T_{\bx} - Q_{\bx}}_{\rm TV}  
        &\le 4\E \norm{\bar Q_{\bx} - Q_{\bx}}_{\rm TV}\, ,
    \end{align*}
    which qualitatively recovers Lemma~\ref{lem:tv_projection}.
\end{remark}

The second inequality in Theorem~\ref{thm:pointwise_projection} can be used in the usual way to deduce concentration bounds for $\norm{T_{\bx} - Q_{\bx}}_{\rm TV}$ when $\bx \sim \pi$. A key feature of this approach is that both terms on the right-hand side of~\eqref{eq:cvx_fn_projection}, in the case of MALA, involve only quantities which measure the discrepancy between the continuous-time Langevin kernel $\bar Q$ and the discretized Langevin proposal $Q$. Therefore, to control the quantity $\norm{T_{\bx} - Q_{\bx}}_{\rm TV}$, it suffices to apply well-established techniques for studying the discretization of SDEs.

Once we show that $\norm{T_{\bx} - Q_{\bx}}_{\rm TV}$ is controlled with high probability, we are then able to apply a conductance argument, similar to Lemma~\ref{lem:standard_conductance_lem}, in order to prove our mixing time bound. We give an in-depth overview of the proof and provide proofs of technical details in Appendix~\ref{scn:upper_bd_proof}.

\section{Lower bound}\label{scn:lower_bd}


It is a standard fact that the mixing time is governed by the inverse of the spectral gap\footnote{By definition, the spectral gap corresponds to the smallest eigenvalue of the Dirichlet form. Hence, for an initial distribution $\mu_0$ that is correlated with the  eigenfunction corresponding to $\lambda$, it follows that $\tau_{\rm mix}(\varepsilon, \mu_0; \sqrt{\chi^2}) = \widetilde{\Omega}(\lambda^{-1})$. See, e.g.,~\cite[Chapter 4]{bakrygentilledoux2014} for a rigorous treatment of spectral theory.}.
Hence, an upper bound on the spectral gap $\lambda$ yields a lower bound on the mixing time.
In addition, we know from Cheeger inequality~\eqref{eq:cheeger} that $\lambda \lesssim \msf C$, where $\msf C$ denotes the conductance of the Markov chain. For these reasons, we identify a lower bound on the mixing time with an upper bound on either the conductance $\msf C$ or the spectral gap $\lambda$. 

To complement our upper bound on the mixing time of MALA, we provide a nearly matching lower bound, thereby settling the question of the dimension dependence of MALA for log-smooth and strongly log-concave targets. To that end, we exhibit a target distribution (in fact a family of distributions) such that the MALA chain with step size $h$ has exponentially small conductance whenever  $h \gg d^{-1/2}$. More precisely, fix $\eta \in (0,1/4)$ and  
define the adversarial target distribution $\pi_\eta$ as a product distribution with potential $V_\eta$ defined by
\begin{align}\label{eq:V_a}
    V_\eta(\bx)
    &= \frac{\|\bx\|^2}{2} - \frac{1}{2d^{2\ee}} \sum_{i=1}^d \cos(d^\ee x_i)
\end{align}
It is not hard to see that $V_\ee$ is $1/2$-strongly convex and $3/2$-smooth. To motivate this choice, recall from \citet[Theorem~1]{roberts1998optimal} that the acceptance probability of MALA  tends to a positive constant as $d\to\infty$ whenever the second moment of the third derivative of the potential is finite and the step size is chosen as $h=\Theta(d^{-1/3})$. The choice $V_\ee$ in~\eqref{eq:V_a} is an example of a smooth and strongly convex potential where this condition is violated asymptotically, therefore suggesting that $h=\Theta(d^{-1/3})$ is too large to prevent the acceptance probability to vanish for large $d$. Our first result below indicates that $h$ should be taken significantly smaller than $d^{-1/3}$; in fact nearly as small as $d^{-1/2}$ when $\eta \approx 1/4$.

In the following theorem, we set $\eta=1/4-\delta$, for some small $\delta>0$.

\begin{theorem}\label{thm:mala_lower_bd_main}
Fix $\delta \in (0, 1/18)$, let $\eta=1/4-\delta$, and let $\msf C$ denote the conductance of the MALA chain with target distribution $\pi_\eta$ and  step size $h$. Then, 
$\msf C \lesssim \exp[-\Omega(d^{4\delta})]$ for any $h \in [  d^{-\frac1 2 + 3\delta}, d^{-\frac{1}{3}}]$. 
 \end{theorem}
Note that as $\delta\searrow 0$, the above theorem shows that MALA must take step sizes which are (essentially) at most of order $d^{-1/2}$.

The next result shows that the spectral gap of MALA is no better than $h$. Together with our upper bound, it implies in particular that the choice $h \approx d^{-1/2}$ is the optimal step size for MALA for a target distribution $\pi_\ee$ and hence, cannot be improved uniformly over the class of distributions with smooth and strongly convex potentials.

\begin{theorem}\label{thm:spectral_gap_upper_bd}
    The spectral gap $\lambda$ of MALA with target distribution $\pi_\ee$ and step size $0< h \le 1$ satisfies  $\lambda \lesssim h $.
\end{theorem}

We give the proofs of these theorems in Appendix~\ref{scn:lower_bd_proof}.


\section{Conclusion}

By establishing the sharp dimension dependence of MALA for smooth and strongly convex potentials, our work parallels well-known trends in optimization~\citep{bubeck2015convex, nesterov2018lectures} and high-dimensional statistics~\citep{tsybakov2009nonparametric, wainwright2019statistics} which seek to characterize the complexity of various learning tasks uniformly over a given function class. It is an interesting open question to extend our results on MALA to other  natural function classes, such as smooth and weakly convex potentials, as well as to other sampling algorithms.

To conclude, we list some specific directions that require further investigations.

\medskip

    
\noindent \textbf{Improved dependence on accuracy and warmness}. A notable weakness of our mixing time bound (Theorem~\ref{thm:mala_upper_bd}) is the dependence on the accuracy parameter and especially the warm start parameter, which are likely artefacts of our analysis. However, we note that in the regime where the step size is as large as $d^{-1/2}$, the conductance profile method of~\cite{chenetal2020hmc} is not enough to remove the effects of a feasible start. Overcoming this challenge may require new tools for controlling the mixing time of a Markov chain.

\medskip{}
\noindent \textbf{Analysis of other Metropolis-Hastings chains}. An interesting feature of Theorem~\ref{thm:mala_upper_bd} is that the majority of the computations involve controlling the discretization error between the continuous-time and discretized Langevin processes, leading to the hope that the vast literature on discretization of SDEs can be leveraged to obtain mixing time bounds for the corresponding Metropolis-Hastings chains. However, a critical component of this program is the choice of a reversible Markov diffusion to which the MALA kernel can be compared via the projection property (Theorem~\ref{thm:pointwise_projection}). As an example, consider the following two settings:
    \begin{enumerate}
        \item Under higher-order smoothness, the diffusion scaling limit of~\cite{roberts1998optimal} suggests that the mixing time of MALA should scale as $d^{1/3}$, using step size $h \approx d^{-1/3}$. Indeed, our computations in Appendix~\ref{scn:gaussian} confirm this prediction for a Gaussian target distribution. However, in this regime, the discretized Langevin proposal is too far from the continuous-time Langevin diffusion for our upper bound strategy to succeed. Thus, in this example, the natural choice of reversible Markov diffusion fails to yield the correct mixing time for MALA.
        \item The underdamped Langevin SDE~\citep{chengetal2018underdamped} is an example of a Markov diffusion which is not reversible. We can consider adding a Metropolis adjustment after a proposal which consists of one step of the discretized underdamped Langevin process. It is not clear that our techniques apply to this example because there does not appear to be a natural \emph{reversible} Markov diffusion with which to compare the resulting Metropolis-adjusted kernel.
    \end{enumerate}
    Despite these obstacles, we believe that there is a wide variety of applications to which our upper bound technique applies, which we leave for future research.


\acks{Sinho Chewi was supported by the Department of Defense (DoD) through the National Defense Science \& Engineering Graduate Fellowship (NDSEG) Program. Kwangjun Ahn was supported by graduate assistantship from the NSF Grant (CAREER: 1846088)
and by the Kwanjeong Educational Foundation. Xiang Cheng was supported by NSF award IIS-1741341. Thibaut Le Gouic was supported by NSF award IIS-1838071. Philippe Rigollet was supported by NSF awards IIS-1838071, DMS-1712596,  and DMS-2022448.
}


\appendix

\section{Proof of the upper bound}\label{scn:upper_bd_proof}
This section presents the proof of Theorem~\ref{thm:mala_upper_bd}.

\subsection{High-level overview of the proof}\label{scn:upper_bd_overview}

The bulk of the proof controls the mixing time in total variation and we use results from Section~\ref{scn:other_distances} to extend it to the other distances.

For the proof, it is technically convenient to work with a refinement of the conductance known as the \emph{$s$-conductance}: for $0 < s < 1/2$, define
\begin{align}\label{eq:C_s}
    \msf C_s
    &:= \inf\Bigl\{\frac{\int_S T(\bx, S^\comp) \, \pi(\D \bx)}{\pi(S)-s} \Bigm\vert S \subseteq \R^d, \; s < \pi(S) \le \frac{1}{2} \Bigr\} \,.
\end{align}

A lower bound on the $s$-conductance translates into an upper bound on the mixing time in total variation distance, via the following lemma.

\begin{lemma}[{\citet[Corollary 1.6]{lovasz1993random}}]
    For any $n \in \N$ and $0 < s < 1/2$, the distribution of the $n$-th iterate $\mu_n$ of the MALA satisfies
    \begin{align*}
        \norm{\mu_n - \pi}_{\rm TV}
        &\le M_0 s + M_0 \exp\bigl( - \frac{\msf C_s^2 n}{2} \bigr),
    \end{align*}
    where $M_0$ is the warm start parameter of $\mu_0$.
\end{lemma}
\begin{corollary}\label{cor:mixing_from_s_conductance}
    Taking $s = \varepsilon/(2M_0)$, it follows that
    \begin{align*}
        \norm{\mu_n - \pi}_{\rm TV} \le \varepsilon \qquad\text{provided that}~n \ge \frac{2}{\msf C_s^2} \ln \frac{2M_0}{\varepsilon}.
    \end{align*}
\end{corollary}

Motivated by the standard conductance lemma (Lemma~\ref{lem:standard_conductance_lem}) and the decomposition~\eqref{eq:TV_decomposition}, in order to bound the $s$-conductance from below we will first bound $\norm{T_{\bx} - Q_{\bx}}_{\rm TV}$, as in Section~\ref{scn:upper_bd}. The outline of the proof is as follows:
\begin{enumerate}
    \item In Section~\ref{scn:pointwise_projection_proof}, we prove the projection properties of MALA (Lemma~\ref{lem:tv_projection} and Theorem~\ref{thm:pointwise_projection}).
    \item In Section~\ref{scn:tv_expectation}, we use the projection property (Lemma~\ref{lem:tv_projection}) along with stochastic calculus to bound the expectation $\E{\norm{T_{\bx} - Q_{\bx}}}_{\rm TV}$ when $\bx \sim \pi$.
    \item In Section~\ref{scn:concentration_of_TV}, we use the pointwise projection property,  together with more stochastic calculus, in order to prove a concentration inequality for $\norm{T_{\bx} - Q_{\bx}}_{\rm TV}$ when $\bx \sim \pi$.
    \item In Section~\ref{scn:conductance_argument}, we use the concentration bound of Section~\ref{scn:concentration_of_TV}, together with ideas from the proof of the standard conductance lemma (Lemma~\ref{lem:standard_conductance_lem}), in order to lower bound the $s$-conductance.
    Together with Corollary~\ref{cor:mixing_from_s_conductance}, it yields the mixing time bound of Theorem~\ref{thm:mala_upper_bd} in total variation distance.
    \item Finally in Section~\ref{scn:other_distances}, we explain how the mixing time bound in total variation distance implies mixing time bounds in other distances between probability measures.
\end{enumerate}

\subsection{Proof of the projection properties}\label{scn:pointwise_projection_proof}

We start with a basic fact about MALA.

\begin{proposition} \label{prop:basics}
Let $Q$ be the proposal kernel and let $T$ be the MALA kernel with proposal $Q$.
Then,
\begin{align*}
    \norm{T_{\bx}-Q_{\bx}}_{\TV}  = \int_{\R^d\setminus\{{\bx}\}} \abs{T(\bx,\by) - Q(\bx,\by)} \, \D \by
    = 1- \int_{\R^d} Q(\bx,\by)A(\bx,\by) \, \D \by.
\end{align*}
\end{proposition}
\begin{proof}
First, since $T_{\bx}$ has an atom at $\bx$ and $Q_{\bx}$ does not, we have
\begin{align*}
   \norm{Q_{\bx}-T_{\bx}}_{\TV} = \frac{1}{2}\,\Bigl( T_{\bx}(\{\bx\}) +\int_{\R^d\setminus\{\bx\}} \abs{T(\bx,\by) - Q(\bx,\by)} \, \D \by \Bigr)  \,.
\end{align*} 
By the definition of the accept-reject step, 
\begin{align*}
    T_{\bx}(\{\bx\}) &= 1- \int_{\R^d\setminus\{\bx\}} T(\bx,\by) \,\D \by=  1- \int_{\R^d} Q(\bx,\by) A(\bx,\by) \,\D \by\,,
\end{align*}
whereas
\begin{align*}
    \int_{\R^d\setminus\{\bx\}} \abs{T(\bx,\by) - Q(\bx,\by)} \, \D \by
    &= 1 - \int_{\R^d} Q(\bx,\by) A(\bx,\by) \, \D \by \,.
\end{align*}
The result follows.
\end{proof}

We now prove the projection properties (Lemma~\ref{lem:tv_projection} and Theorem~\ref{thm:pointwise_projection}).

\medskip{}

\begin{proof}[Proof of Lemma~\ref{lem:tv_projection}]
Since the transition kernel $\bar Q$ corresponding to the continuous-time Langevin diffusion is reversible with stationary distribution $\pi$, it follows from~\citet{billeradiaconis2001mhprojection} that
\begin{align*}
    &\iint\displaylimits_{(\R^d\times\R^d) \setminus \Delta } \abs{T(\bx,\by) - Q(\bx,\by)} \, \pi(\D \bx) \, \D \by
    \le \iint\displaylimits_{(\R^d\times\R^d) \setminus \Delta} \abs{\bar Q(\bx,\by) - Q(\bx,\by)} \, \pi(\D \bx) \, \D \by\,,
\end{align*}
where $\Delta=\{(\bx,\by) \in \R^d \times \R^d \,:\, \bx=\by\}$.
Since $Q_{\bx}$ and $\bar Q_{\bx}$ have no atoms, the right-hand side is equal to $2\E_{\bx \sim \pi}\norm{\bar Q_{\bx} - Q_{\bx}}_{\rm TV}$.
On the other hand, the left-hand side is equal to $\E_{\bx\sim \pi}\norm{T_{\bx} - Q_{\bx}}_{\rm TV}$ due to Proposition~\ref{prop:basics}.
\end{proof}

\begin{proof}[Proof of Theorem~\ref{thm:pointwise_projection}]
For any $\bx$, we have
\begin{align*}
    &\norm{T_{\bx} - Q_{\bx}}_{\rm TV}
    = \int \{1-A(\bx,\by)\} \,Q(\bx,\by) \, \D \by
    = \int \Bigl[1- \Bigl(1 \wedge \frac{\pi(\by) Q(\by,\bx)}{\pi(\bx) Q(\bx,\by)} \Bigr)\Bigr] \, Q(\bx,\by) \, \D \by \\
    &\qquad \le \int \Bigl\lvert 1- \frac{\pi(\by) Q(\by,\bx)}{\pi(\bx) Q(\bx,\by)} \Bigr\rvert \, Q(\bx,\by) \, \D \by \\
    &\qquad \le \int \Bigl\lvert 1- \frac{\pi(\by) \bar Q(\by,\bx)}{\pi(\bx) Q(\bx,\by)} \Bigr\rvert \, Q(\bx,\by) \, \D \by  + \int \frac{\pi(\by) \bar Q(\by, \bx)}{\pi(\bx)} \, \bigl\lvert \frac{Q(\by,\bx)}{\bar Q(\by,\bx)} - 1\bigr\rvert \, \D \by.
\end{align*}
Observe that the first term is given by
$$
\int \Bigl\lvert 1- \frac{\pi(\by) \bar Q(\by,\bx)}{\pi(\bx) Q(\bx,\by)} \Bigr\rvert \, Q(\bx,\by) \, \D \by
    = \int \Bigl\lvert Q(\bx,\by) - \frac{\pi(\by) \bar Q(\by,\bx)}{\pi(\bx)} \Bigr\rvert \, \D \by 
    = 2 \, \norm{Q_{\bx} - \bar Q_{\bx}}_{\rm TV}\,,
$$
where in the second identity, we used the reversibility of $\bar Q$. 
This concludes the proof of the first inequality.

We now deduce the second inequality from the first. Using monotonicity and convexity of $\Phi$ respectively, we get, 
    \begin{align*}
        &\E\Phi(\norm{T_{\bx} - Q_{\bx}}_{\rm TV})
        \le \E \Phi\Bigl( 2\,\norm{\bar Q_{\bx} - Q_{\bx}}_{\rm TV} + \int \frac{\pi(\by) \bar Q(\by, \bx)}{\pi(\bx)} \, \bigl\lvert \frac{Q(\by,\bx)}{\vphantom{\big|}\bar Q(\by,\bx)} - 1\bigr\rvert \, \D \by \Bigr) \\
        &\qquad \le \frac{1}{2} \E \Phi(4 \, \norm{\bar Q_{\bx} - Q_{\bx}}_{\rm TV}) + \frac{1}{2} \E \Phi\Bigl(2 \int \frac{\pi(\by) \bar Q(\by, \bx)}{\pi(\bx)} \, \bigl\lvert \frac{Q(\by,\bx)}{\vphantom{\big|}\bar Q(\by,\bx)} - 1\bigr\rvert \, \D \by \Bigr)\,,
    \end{align*}
    where we take expectation with respect to $\bx \sim \pi$. Next, nothing that $\int \pi(\by) \bar Q(\by, \bx) \, \D \by = \pi(\bx)$, we apply Jensen's inequality to yield
    \begin{align*}
        &\E \Phi\Bigl(2 \int \frac{\pi(\by) \bar Q(\by, \bx)}{\pi(\bx)} \, \bigl\lvert \frac{Q(\by,\bx)}{\vphantom{\big|}\bar Q(\by,\bx)} - 1\bigr\rvert \, \D \by \Bigr) \\
        &\qquad = \int \Phi\Bigl(2 \int \frac{\pi(\by) \bar Q(\by, \bx)}{\pi(\bx)} \, \bigl\lvert \frac{Q(\by,\bx)}{\vphantom{\big|}\bar Q(\by,\bx)} - 1\bigr\rvert \, \D \by \Bigr) \, \pi(\bx) \, \D \bx \\
        &\qquad \le \iint \Phi\bigl(2\,\bigl\lvert \frac{Q(\by,\bx)}{\vphantom{\big|}\bar Q(\by,\bx)} - 1\bigr\rvert\bigr) \, \pi(\by) \bar Q(\by,\bx) \, \D \bx \, \D \by \\
        &\qquad = \iint \Phi\bigl(2\,\bigl\lvert \frac{Q(\bx,\by)}{\vphantom{\big|}\bar Q(\bx,\by)} - 1\bigr\rvert\bigr) \, \pi(\bx) \bar Q(\bx,\by) \, \D \bx \, \D \by\,,
    \end{align*}
    where we switched $\bx$ and $\by$ in the notation of the last line.
\end{proof}

\subsection{Expectation of the total variation}\label{scn:tv_expectation}

We now bound $\E\norm{T_{\bx} - Q_{\bx}}_{\rm TV}$  when $\bx \sim \pi$ using the projection property (Lemma~\ref{lem:tv_projection}). Akin to prior work such as~\cite{dalalyan2012sparse}, our primary tool to analyze the discretization of the Langevin diffusion is the Girsanov theorem from stochastic calculus \citep[see, e.g.][for classical treatments]{legall2016stochasticcalc,stroockMultidimensionalDiffusionProcesses2007}.

\begin{lemma}[Girsanov theorem]\label{lem:girsanov}
    Let $\bar{\mb Q}_{\bx}$ denote the probability measure on path space induced by the solution ${(\bar \bX_t)}_{t \in [0, h]}$ of the continuous-Langevin diffusion SDE~\eqref{eq:langevin_sde} started at $\bx$ and run for time $h>0$.
    Moreover, let $\mb Q_{\bx}$ denote the probability measure on path space induced by the solution of the following SDE with constant drift
    \begin{align*}
        \D \bX_t
        &= -\nabla V(\bx) \, \D t + \sqrt 2 \, \D \bB_t, \qquad \bX_0 = \bx.
    \end{align*}
    Then, $\mb Q_{\bx}$ is absolutely continuous with respect to $\bar{\mb Q}_{\bx}$ and has density given by  Radon-Nikodym derivative:
    \begin{align*}
        \frac{\D\mb Q_{\bx}}{\vphantom{\big|}\D \bar{\mb Q}_{\bx}}\bigl({(\bar \bX_t)}_{t}\bigr)  
       = \exp\Bigl[ \frac{1}{\sqrt 2} \int_0^h \langle \nabla V(\bar \bX_t) - \nabla V(\bx), \D \bB_t \rangle - \frac{1}{4} \int_0^h  \norm{\nabla V(\bar \bX_t) - \nabla V(\bx)}^2 \, \D t\Bigr].
    \end{align*}
\end{lemma}
\begin{proof}
    See the proof of Proposition 2 in~\cite{dalalyan2012sparse}.
\end{proof}

In the following lemma, we use Lemma~\ref{lem:ito_distance}.

\begin{lemma}\label{lem:expectation_TV}
Assume $h \le 1/(3\beta^{4/3})$.
For any $\bx \in \R^d$,
\begin{align*}
\norm{ \bar Q_{\bx}- Q_{\bx}}_{\rm TV}  \leq \frac{1}{2}\beta h \sqrt{d+\beta^{2/3} \,\norm{\bx}^2}\,.
\end{align*}
\end{lemma}
\begin{proof}
Let $\ept$ denote the function that maps a continuous curve ${(y_t)}_{t\in[0,h]}$ in $\R^d$ to its endpoint: $\ept({(y_t)}_{t\in [0,h]}):= y_h$.
Then, it is clear that
\begin{align*}
    Q_{\bx}  = \ept_{\#} \mb Q_{\bx} \qquad \text{and}\qquad \bar{Q}_{\bx}  = \ept_{\#} \bar{\mb Q}_{\bx} \,,
\end{align*}
where the notation $f_\# \mu$ denotes the pushforward of a measure $\mu$ under the mapping $f$.
On the one hand, it follows from the data processing inequality that 
\begin{align*}
    \KL(\bar Q_{\bx} ~\|~ Q_{\bx}) = \KL(  \ept_{\#} \bar{\mb Q}_{\bx}~\|~\ept_{\#} \mb Q_{\bx} )\leq \KL(    \bar{\mb Q}_{\bx}~\|~\mb Q_{\bx})\,.  
\end{align*}
On the other hand, the Girsanov theorem (in the form of Lemma~\ref{lem:girsanov}) implies that 
\begin{align*}
    \KL(    \bar{\mb Q}_{\bx}~\|~\mb Q_{\bx})&=-\E \ln \frac{{\D \mb Q}_{\bx}}{\D \bar{\mb Q}_{\bx}}(\bar{\bX}_t) = \frac{1}{4}  \int_0^{h} \E[\norm{\nabla V(\bar{\bX}_t) - \nabla V(\bx)}^2] \, \D t\\
    &\leq  \frac{\beta^2}{4}  \int_0^{h} \E[\norm{ \bar{\bX}_t - \bx}^2] \, \D t
    \le \frac{3 \beta^2 h^2 \, (d+\beta^{2/3} \,\norm{\bx}^2)}{8} \,,
\end{align*}
where we used the $\beta$-smoothness of $V$ and Lemma~\ref{lem:ito_distance}.
Now applying Pinsker's inequality, we obtain the desired inequality.
\end{proof}

It follows from Lemma~\ref{lem:expectation_TV} that when $\bx \sim \pi$, we get
\begin{align}\label{eq:expectation_Qbar_Q}
    \begin{aligned}
    \E\norm{\bar Q_{\bx} - Q_{\bx}}_{\rm TV}
    &\le \frac{1}{2} \beta h \E\sqrt{d+\beta^{2/3} \,\norm{\bx}^2}
    \le \frac{1}{2} \beta h \sqrt{d + \beta^{2/3} \E[\norm{\bx}^2]} 
    \lesssim \beta^{4/3} h \sqrt{\frac{d}{\alpha}} \,,
    \end{aligned}
\end{align}
where we used the second moment bound of Lemma~\ref{lem:properties_of_strong_log_concave}. Together with Lemma~\ref{lem:tv_projection}, it yields
\begin{align}\label{eq:expectation_tv}
    \E\norm{T_{\bx} - Q_{\bx}}_{\rm TV}
    &\le 2\E\norm{\bar Q_{\bx} - Q_{\bx}}_{\rm TV}
    \lesssim \beta^{4/3} h \sqrt{\frac{d}{\alpha}} \,.
\end{align}

We conclude this section with a concentration inequality which we use later in the argument.

\begin{lemma}\label{lem:TV_concentration_1}
    Assume $h \le 1/(3\beta^{4/3})$ and let $\bx \sim \pi$.
    For any $\delta > 0$, with probability at least $1-\delta$,
    \begin{align*}
        \norm{\bar Q_{\bx} - Q_{\bx}}_{\rm TV}
        &\lesssim \beta^{4/3} h \sqrt{\frac{d + \log(1/\delta)}{\alpha}}.
    \end{align*}
\end{lemma}
\begin{proof}
    Let $f(\bx) := \frac{1}{2} \beta^{4/3} h \sqrt{d+\norm{\bx}^2}$.
    Then,
    \begin{align*}
        \norm{\nabla f(\bx)}
        &= \frac{\beta^{4/3} h \, \norm{\bx}}{2\sqrt{d+\norm{\bx}^2}}
        \le \frac{1}{2} \beta^{4/3} h.
    \end{align*}
    Thus, $f(\bx)$ is $\frac{1}{2} \beta^{4/3} h$-Lipschitz, and it follows from sub-Gaussian concentration (Lemma~\ref{lem:properties_of_strong_log_concave}) that with probability at least $1-\delta$,
    \begin{align*}
        f(\bx)
        &\le \E f(\bx) + \beta^{4/3} h \sqrt{\frac{1}{2\alpha} \ln\frac{1}{\delta}}.
    \end{align*}
    We have calculated $\E f(\bx) \lesssim \beta^{4/3} h\sqrt{d/\alpha}$ in~\eqref{eq:expectation_Qbar_Q}, and the result now follows from the pointwise bound in Lemma~\ref{lem:expectation_TV}.
\end{proof}

\subsection{Concentration of the total variation}\label{scn:concentration_of_TV}
Equation~\eqref{eq:expectation_tv} provides a control the total variation distance between the MALA kernel and the proposal \emph{in expectation}. The main result of this section is an extension of this result to a control \emph{with high probability} captured in the following proposition.
\begin{proposition}\label{prop:tv_concentration}
Fix $c_0 > 0$ and $0 < s < 1/2$.
Then, there exists a constant $c_1 > 0$, depending only on $c_0$, such that with step size
\begin{align*}
    h
    &= \frac{c_1 \alpha^{1/2}}{\beta^{4/3} d^{1/2} \log(d\kappa/s)} \,,
\end{align*}
the following holds with probability at least $1-c_0 s\sqrt h$, 
\begin{align*}
    \normb{T_{\bx} - Q_{\bx}}_{\TV } \leq \frac 1 6 \,.
\end{align*}
\end{proposition}
The idea of the proof is to use the pointwise projection of Theorem~\ref{thm:pointwise_projection}, and to obtain high probability bounds for each of the two terms in \eqref{eq:pointwise_TV}. An upper bound for the first term follows directly from Lemma~\ref{lem:TV_concentration_1}. To control the second term, we will first obtain a bound on its moments.
\begin{lemma}\label{lem:moment_bdd}
    Let $k \ge 1$ be any integer.
    Suppose that
    \begin{align*}
        h
        \le  \frac{\alpha^{1/2} }{C\beta^{4/3} d^{1/2} k}\,,\qquad\text{for a sufficiently large absolute constant}~C > 0\,.
    \end{align*}
    Then, it holds that
    \begin{align*}
        \Bigl\{\E_{\bx \sim \pi}\Bigl[\Bigl\lvert\int \frac{\pi(\by) \bar Q(\by, \bx)}{\pi(\bx)} \, \bigl\lvert \frac{Q(\by,\bx)}{\vphantom{\big|}\bar Q(\by,\bx)} - 1\bigr\rvert \, \D \by \Bigr\rvert^k\Bigr]\Bigr\}^{1/k} \lesssim \alpha^{-1/4} \beta h \sqrt k \, (\sqrt d + \sqrt k)\,.
    \end{align*}
\end{lemma}
The proof, given in Appendix~\ref{sec:moment_bdd}, uses extensively tools from stochastic calculus. We remark that the quantity in Lemma~\ref{lem:moment_bdd} can be interpreted as a bound on the R\'enyi divergence between the discretized and continuous Langevin processes. A similar result has appeared as~\cite[Corollary 11]{ganeshtalwar2020renyi}.

We are now in a position to prove Proposition~\ref{prop:tv_concentration}.  

\medskip{}
\begin{proof}[Proof of Proposition~\ref{prop:tv_concentration}]
    Assume that the step size $h$ is small enough so that Lemmas~\ref{lem:TV_concentration_1} and~\ref{lem:moment_bdd} both hold. 
    More specifically, since the requirement of Lemma~\ref{lem:moment_bdd} is more stringent than that of Lemma~\ref{lem:TV_concentration_1}, so we can simply impose  $h\leq \frac{\alpha^{1/2} }{C\beta^{4/3} d^{1/2} k} $ for a sufficiently large absolute constant $C>0$.

  From Lemma~\ref{lem:TV_concentration_1} with $\delta = c_0 s\sqrt h/2$, there exists a constant $C_1 > 0$ such that with probability at least $1-c_0 s\sqrt h/2$,
    \begin{align*}
        \norm{\bar Q_{\bx} - Q_{\bx}}_{\rm TV}
        &\le \frac{C_1 \beta^{4/3} h}{2\sqrt\alpha} \sqrt{d + \ln \frac{2}{c_0 s\sqrt h}} \,.
    \end{align*}
    From Lemma~\ref{lem:moment_bdd} and Markov's inequality, there exists a constant $C_2 > 0$ such that for any $\delta > 0$, with probability at least $1-\delta$,
    \begin{align*}
        \int \frac{\pi(\by) \bar Q(\by, \bx)}{\pi(\bx)} \, \bigl\lvert \frac{Q(\by,\bx)}{\vphantom{\big|}\bar Q(\by,\bx)} - 1\bigr\rvert \, \D \by
        &\le C_2 \alpha^{-1/4} \beta h \sqrt k \, (\sqrt d + \sqrt k)\, \delta^{-1/k} \,.
    \end{align*}
    Taking $k \sim \ln \frac{2}{c_0 s \sqrt h}$ and $\delta = c_0 s\sqrt h/2$, we have $\delta^{-1/k} =\Theta(1)$ and hence
    \begin{align*}
        \int \frac{\pi(\by) \bar Q(\by, \bx)}{\pi(\bx)} \, \bigl\lvert \frac{Q(\by,\bx)}{\vphantom{\big|}\bar Q(\by,\bx)} - 1\bigr\rvert \, \D \by
        &\le C_2 \alpha^{-1/4} \beta h \sqrt{\ln \frac{2}{c_0 s\sqrt h}} \, \Bigl(\sqrt d + \sqrt{\ln \frac{2}{c_0 s\sqrt h}}\Bigr)\,.
    \end{align*}
    Combining these two inequalities with the pointwise projection property (Theorem~\ref{thm:pointwise_projection}), it follows that with probability at least $1-c_0 s\sqrt h$,
    \begin{align} \label{ineq:tv_concen_final}
        \norm{T_{\bx} - Q_{\bx}}_{\rm TV}
        &\le \frac{C_1 \beta^{4/3} h}{\sqrt\alpha} \sqrt{d + \ln \frac{2}{c_0 s\sqrt h}} + C_2 \alpha^{-1/4} \beta h \sqrt{\ln \frac{2}{c_0 s\sqrt h}} \, \Bigl(\sqrt d + \sqrt{\ln \frac{2}{c_0 s\sqrt h}}\Bigr) \,.
    \end{align}
    If we choose the constant $c_1 > 0$ small enough, then choosing the step size as in the statement of Proposition~\ref{prop:tv_concentration}, i.e., $h= \frac{c_1 \alpha^{1/2}}{\beta^{4/3} d^{1/2} \log(d\kappa/s)}$, makes the both terms in the left-hand side of \eqref{ineq:tv_concen_final} less than $1/12$. This completes the proof of Proposition~\ref{prop:tv_concentration}. 
\end{proof}

\subsubsection{Proof of Lemma~\ref{lem:moment_bdd}}\label{sec:moment_bdd}
We now prove the moment upper bound (Lemma~\ref{lem:moment_bdd}). Since $\int \pi(\by)\bar Q(\by,\bx) \, \D \by = \pi(\bx)$, we can apply Jensen's inequality to get
\begin{align*}
    \int \pi(\bx) \, \Bigl\lvert \int \frac{\pi(\by) \bar Q(\by,\bx)}{\pi(\bx)} \, \bigl\lvert \frac{Q(\by,\bx)}{\vphantom{\big|}\bar Q(\by,\bx)} - 1 \bigr\rvert \, \D \by \Bigr\rvert^k \, \D \bx
    & \le \iint \pi(\by) \bar Q(\by,\bx) \, \bigl\lvert \frac{Q(\by,\bx)}{\vphantom{\big|}\bar Q(\by,\bx)} - 1 \bigr\rvert^k \, \D \bx \, \D \by \\
    & = \int \Bigl( \int \bigl\lvert \frac{Q(\bx,\by)}{\vphantom{\big|}\bar Q(\bx,\by)} - 1 \bigr\rvert^k \, \bar Q(\bx, \D \by) \Bigr) \, \pi(\D \bx)\,,
\end{align*}
where we switched $\bx$ and $\by$ in the last line.
The inner integral equals the $f$-divergence $D_f(Q_{\bx} \mmid \bar Q_{\bx})$, with $f(\bx) := \abs{x-1}^k$.
Recall the definitions of $\bar{\mb Q}_{\bx}$ and $\mb Q_{\bx}$ in Lemma~\ref{lem:girsanov}.
Hence we may apply the data processing inequality and bound the above by
\begin{align}\label{exp:k_moment}
  F_k:=  \int \Bigl( \int \bigl\lvert \frac{\D \mb Q_{\bx}}{\vphantom{\big|}\D \bar{\mb Q}_{\bx}} - 1 \bigr\rvert^k \, \D \bar{\mb Q}_{\bx} \Bigr) \, \pi(\D \bx)\,.
\end{align}
Recall from Lemma~\ref{lem:girsanov} that 
\begin{align*}
    \frac{\D \mb Q_{\bx}}{\vphantom{\big|}\D \bar{\mb Q}_{\bx}}(\bar\bX)
    &= \exp H_h\,,
\end{align*}
where for $t\ge 0$,
\begin{align*}
    H_t
    &:= \frac{1}{\sqrt 2} \int_0^t\langle \nabla V(\bar \bX_s) - \nabla V(\bx), \D \bB_s \rangle - \frac{1}{4} \int_0^t \norm{\nabla V(\bar \bX_s) - \nabla V(\bx)}^2 \, \D s\,.
\end{align*} 
Applying It\^o's formula to ${(H_t)}_{t\ge 0}$ and the function $\exp$, we deduce that
\begin{align*}
    \exp H_h - 1
    &= \frac{1}{\sqrt 2} \int_0^h (\exp H_t) \, \langle \nabla V(\bar \bX_t) - \nabla V(\bx), \D \bB_t\rangle.
\end{align*}
In what follows, $\bE_{\bx}$ denotes the expectation under $\bar{\mb Q}_{\bx}$ (the measure under which $\bar{\bX}$ is a continuous-time Langevin diffusion).
Also, we will use the letter $C$ to denote a numerical constant which may change from line to line.
Based on the upper bound \eqref{exp:k_moment} on the $k$-th moment, we wish to estimate
\begin{align*}
  F_k&=\bE_{\bx}[\abs{\exp H_h - 1}^k]
    = \frac{1}{2^{k/2}} \bE_{\bx}\Bigl[ \Bigl\lvert \int_0^h (\exp H_t) \, \langle \nabla V(\bar \bX_t) - \nabla V(\bx), \D \bB_t\rangle \Bigr\rvert^k \Bigr] \\
    &\le {(Ck)}^{k/2} \bE_{\bx}\Bigl[\Bigl\lvert \int_0^h \exp(2H_t) \, \norm{\nabla V(\bar\bX_t) - \nabla V(\bx)}^2  \, \D t \Bigr\rvert^{k/2}\Bigr]\\
    \intertext{where the last line is the Burkholder-Davis-Gundy inequality with optimal constants \citep{burkholderDistributionFunctionInequalities1973, davis1976lpnorm}. Together with the Cauchy-Schwarz inequality and H\"older's inequality, it yields}
F_k
    &\le {(C\beta^2 k)}^{k/2} \bE_{\bx}\Bigl[ \Bigl\lvert \int_0^h \exp(4H_t) \, \D t \Bigr\rvert^{k/4} \, \Bigl\lvert \int_0^h \|\bar \bX_t - \bx\|^4 \, \D t\Bigr\rvert^{k/4}\Bigr] \\
    &\le {(C\beta^2 k)}^{k/2} \sqrt{\bE_{\bx}\Bigl[\Bigl\lvert \int_0^h \exp(4H_t) \, \D t \Bigr\rvert^{k/2}\Bigr] \bE_{\bx}\Bigl[ \Bigl\lvert \int_0^h \norm{\bar \bX_t - \bx}^4 \, \D t \Bigr\rvert^{k/2}\Bigr]} \\
    &\le {(C\beta^2 k)}^{k/2} \, h^{k/2-1} \underbrace{\sqrt{\Bigl(\bE_{\bx}\int_0^h \exp(2kH_t) \, \D t \Bigr)}}_{\cAA} \,\underbrace{\sqrt{\Bigl(\bE_{\bx} \int_0^h \norm{\bar \bX_t - \bx}^{2k} \, \D t \Bigr)}}_{\cBB}\,.
\end{align*}
We will control the two terms separately, starting with the first term \tcAA.

\begin{lemma}\label{lem:exp_H_t}
    Let $0 \le t \le h \le 1/(20\beta k)$.
    Then,
    \begin{align*}
        \bE_{\bx} \exp(2kH_t) \le \exp(96\beta^4 h^3 k^2 \, \norm{\bx}^2 + 576 \beta^2 dh^2 k^2).
    \end{align*}
\end{lemma}
\begin{proof}
Recall the following fact, which follows from It\^o's lemma~\citep[Theorem 5.10]{legall2016stochasticcalc}: for any adapted process ${(\bs Z_s)}_{s\ge 0}$, we have 
$$
\bE_{\bx} \exp(\int_0^t  \langle \bs Z_s, \D \bB_s \rangle - \frac{1}{2} \int_0^t\|\bs Z_s\|^2 \, \D s) = 1\,.
$$
Together with the Cauchy-Schwarz inequality, it yields
\begin{align*}
    &\bE_{\bx}\exp(2 k H_t)\\
    &\qquad = \bE_{\bx} \exp\Bigl[\sqrt 2 k \int_0^t\langle \nabla V(\bar \bX_s) - \nabla V(\bx), \D \bB_s \rangle - \frac{k}{2} \int_0^t \norm{\nabla V(\bar \bX_s) - \nabla V(\bx)}^2 \, \D s\Bigr] \\
    &\qquad =\bE_{\bx} \exp\Bigl[\sqrt 2 k \int_0^t\langle \nabla V(\bar \bX_s) - \nabla V(\bx), \D \bB_s \rangle  \\
    &\qquad\qquad\qquad\qquad\qquad{}+ \bigl( - 4k^2 + 4k^2 - \frac{k}{2}\bigr) \int_0^t \norm{\nabla V(\bar \bX_s) - \nabla V(\bx)}^2 \, \D s\Bigr] \\
    &\qquad \le \sqrt{\bE_{\bx} \exp\Bigl[8k^2 \int_0^t \norm{\nabla V(\bar \bX_s) - \nabla V(\bx)}^2 \, \D s \Bigr]} \\
    &\qquad \le \sqrt{\bE_{\bx} \exp\Bigl[8\beta^2 k^2 \int_0^t \norm{\bar \bX_s - \bx}^2 \, \D s \Bigr]}
    \le \sqrt{\bE_{\bx} \exp\Bigl[8\beta^2 hk^2 \sup_{s\in [0,h]} \norm{\bar \bX_s - \bx}^2 \Bigr]}\,.
\end{align*}
In order to upper bound the above quantity, we develop the following bound on the moment generating function of $\sup_{s\in [0,h]} \norm{\bar \bX_s - \bx}^2$.

\begin{restatable}{lemma}{lemsupcontrol}
\label{lem:sup_xt_control}
    Assume $h \le 1/(2\beta)$.
    For $0 < \lambda < 1/(24h)$,
    \begin{align*}
        \bE_{\bx} \exp\bigl(\lambda \sup_{t\in [0,h]}{\norm{\bar\bX_t - \bx}^2}\bigr)
        &\le \exp\bigl(12\beta^2 \lambda h^2 \, \norm{\bx}^2 + d\ln \frac{1+24h\lambda}{1-24h\lambda}\bigr).
    \end{align*}
\end{restatable}
\begin{proof}
    The proof is deferred to \S\ref{sec:pf_aux_stochastic}.
\end{proof}
We use  Lemma~\ref{lem:sup_xt_control} with $\lambda := 8\beta^2 hk^2$. In order to satisfy the preconditions of  Lemma~\ref{lem:sup_xt_control}, we impose the restriction $h \le \frac{1}{14\beta k}$. Then, it follows that
\begin{align*}
    \bE_{\bx}\exp(2 k H_t)
    &\le \exp\bigl(96\beta^4  h^3 k^2 \, \norm{\bx}^2 + d\ln \frac{1+192\beta^2 h^2 k^2}{1 - 192\beta^2 h^2 k^2}\bigr) \\
    &\le \exp(96\beta^4 h^3 k^2 \, \norm{\bx}^2 + 576 \beta^2 dh^2 k^2) \,,
\end{align*}
where the last inequality is $\ln \frac{1+x}{1-x} \le 3x$, which holds provided $x \le 1/2$; this is valid provided $h\le \frac{1}{20\beta k}$.
This is our desired bound.
\end{proof}

Hence, from Lemma~\ref{lem:exp_H_t}, we obtain
\begin{align*}
    \ncAA \leq  \sqrt{h\exp(96\beta^4 h^3 k^2 \, \norm{\bx}^2 + 576 \beta^2 dh^2 k^2)}\,.
\end{align*}

Next, we estimate \tcBB. In fact,  Lemma~\ref{lem:sup_xt_control} together with standard moment bounds under sub-exponential concentration (e.g.~\cite[Proposition 2.7.1]{vershynin2018highdimprob}) gives
\begin{align*}
        \bE_{\bx} \sup_{t\in [0,h]}{\norm{\bar \bX_t - \bx}^{2k}}
        &\le C^k \, (\beta^k h^{2k} \, \norm{\bx}^{2k} + d^k h^k + h^k k^k)\,,
\end{align*}
where $C>0$ is a numerical constant. See   Corollary~\ref{cor:moments_of_langevin}  in \S\ref{sec:pf_aux_stochastic} for details. Hence, it holds that 
\begin{align*} 
    \ncBB=\int_0^h \bE_{\bx}[\|\bar \bX_t - \bx\|^{2k}] \, \D t \le C^k h \, (\beta^k h^{2k} \, \norm{\bx}^{2k} + d^k h^k + h^k k^k). 
\end{align*}
Hence,
\begin{align*}
    \eqref{exp:k_moment} &\leq {(C\beta^2 k)}^{k/2}  h^{k/2-1} \times\ncAA \times \ncBB\\ 
    &\le {(C\beta^2 k)}^{k/2}  h^{k/2-1} \times h^{1/2}\exp(48\beta^4 h^3 k^2 \, \norm{\bx}^2 + 288 \beta^2 dh^2 k^2)  \\
    &\qquad\qquad{} \times   \sqrt{C^k h \, (\beta^k h^{2k} \, \norm{\bx}^{2k} + d^k h^k + h^k k^k)} \\
 &\le {(C^2\beta^2 h k)}^{k/2}  \exp( 288 \beta^2 dh^2 k^2)  \\
    &\qquad\qquad{} \times \exp(48\beta^4 h^3 k^2\,\norm{\bx}^2)  \sqrt{C^k h \, (\beta^k h^{2k} \, \norm{\bx}^{2k} + d^k h^k + h^k k^k)}.
\end{align*}
Next, we take the expectation w.r.t.\ $\bx \sim \pi$ and use Cauchy-Schwarz:
\begin{align*}
    &\E_{\bx \sim \pi}\bE_{\bx}[\abs{\exp H_h - 1}^k] \\
    &\qquad \le {(C\beta^2 hk)}^{k/2} \exp(288\beta^2 dh^2 k^2) \\
    &\qquad\qquad\qquad\times{} \sqrt{\E_{\bx \sim \pi}\exp(96\beta^4 h^3 k^2 \, \norm{\bx}^2) \E_{\bx \sim \pi}[\beta^k h^{2k} \, \norm{\bx}^{2k} + d^{k} h^{k} + h^{k} k^{k}]} \,.
\end{align*}
For the two terms involving exponentials: the first will be bounded by a numerical constant provided that $h \le \frac{1}{C\beta k \sqrt d}$, and using concentration properties of $\pi$ (see e.g.\ Lemma~\ref{lem:properties_of_strong_log_concave}), the second will be bounded provided $h \le \frac{\alpha^{1/3}}{C \beta^{4/3} d^{1/3} k^{2/3}}$. Taking this to be the case, the moment bounds in Lemma~\ref{lem:properties_of_strong_log_concave} now imply the bound
\begin{align*}
    & \E_{\bx \sim \pi}\bE_{\bx}[\abs{\exp H_h - 1}^k] \\
    &\qquad \le {(C\beta^2 hk)}^{k/2} \times (\alpha^{-k/2} \beta^{k/2} d^{k/2} h^k + \alpha^{-k/2} \beta^{k/2} h^k k^{k/2} + d^{k/2} h^{k/2} + h^{k/2} k^{k/2}) \,.
\end{align*}
Taking $k$-th roots,
\begin{align*}
    &{(\E_{\bx \sim \pi}\bE_{\bx}[\abs{\exp H_h - 1}^k])}^{1/k} \\
    &\qquad \lesssim \beta \sqrt{hk} \times (\alpha^{-1/2} \beta^{1/2} d^{1/2} h + \alpha^{-1/2} \beta^{1/2} h k^{1/2} + d^{1/2} h^{1/2} + h^{1/2} k^{1/2}) \\
    &\qquad \lesssim \alpha^{-1/4} \beta h \sqrt k \, (\sqrt d + \sqrt k),
\end{align*}
provided that $h \le \alpha^{1/2}/\beta$. 
This concludes the proof.


\subsection{Conductance argument}\label{scn:conductance_argument}

In this section, we use the results from the previous sections in order to prove a lower bound on the $s$-conductance.
The argument is similar to the proof of the standard conductance lemma (Lemma~\ref{lem:standard_conductance_lem}).

Towards the goal of applying the bound on the mixing time via $s$-conductance given in Corollary~\ref{cor:mixing_from_s_conductance}, we take $s := \varepsilon/(2M_0)$, and we choose the step size
\begin{align}\label{eq:cond_arg_step_size}
    h
    &= \frac{c_1 \alpha^{1/2}}{\beta^{4/3} d^{1/2} \log(d\kappa/s)}
\end{align}
as in Proposition~\ref{prop:tv_concentration}.
Then, Proposition~\ref{prop:tv_concentration} guarantees the existence of an event $\eve$ with probability $\pi(\eve) \ge 1-c_0 s\sqrt h$ such that
\begin{align*}
    \bx \in \eve \implies \normb{T_{\bx} - Q_{\bx}}_{\TV } \leq \frac 1 6\,.
\end{align*}

Let $S$ be a measurable subset of $\R^d$ with $s \le \pi(S) \le 1/2$.
Define the following subsets:
\begin{align*}
    S_1
    &:= \bigl\{\bx\in S \bigm\vert T(\bx, S^\comp) \le \frac{1}{4}\bigr\}\,,  &\text{bad set 1}\\
    S_2
    &:= \bigl\{\bx\in S^\comp \bigm\vert T(\bx,S) \le \frac{1}{4} \bigr\}\,, &\text{bad set 2}\\
    S_3&:= {(S_1\cup S_2)}^\comp. &\text{good set}
\end{align*}

If $\pi(S_1) < \pi(S)/2$ or $\pi(S_2) < \pi(S^\comp)/2$, then may conclude from reversibility of the MALA kernel $T$ that
\begin{align*}
    \int_S T(\bx, S^\comp) \, \pi(\D \bx)
    &= \frac{1}{2} \Bigl(\int_S T(\bx, S^\comp) \, \pi(\D \bx) + \int_{S^\comp} T(\bx, S) \, \pi(\D \bx)\Bigr)
    \ge \frac{1}{2} \cdot \frac{\pi(S)}{2} \cdot \frac{1}{4}
    = \frac{\pi(S)}{16} \,.
\end{align*}
Therefore, for the purpose of proving a lower bound on the $s$-conductance, we may assume that $\pi(S_1) \wedge \pi(S_2) \ge \pi(S)/2$.

Now we consider $\bx\in E \cap S_1$ and $\by\in E \cap S_2$.
From the definitions of $S_1$ and $S_2$, it follows that
\[\norm{T_{\bx} - T_{\by}}_{\rm TV} \geq \frac 1 2\,.\]
Since $\bx,\by\in E$, we also have
\[\norm{T_{\bx} - Q_{\bx}}_{\rm TV} \wedge \norm{T_{\by} - Q_{\by}}_{\rm TV} \le \frac 1 6\,.\]
Thus, using the decomposition~\eqref{eq:TV_decomposition},
\begin{align*}
    \frac 1 2 
    &\leq \norm{T_{\bx} - T_{\by}}_{\rm TV}
    \leq \norm{T_{\bx} - Q_{\bx}}_{\rm TV} + \norm{Q_{\bx} - Q_{\by}}_{\rm TV} + \norm{T_{\by} - Q_{\by}}_{\rm TV}\\
    &\leq \frac 1 6 + \frac{\norm{\bx - \by}}{\sqrt{2h}} + \frac 1 6\,,
\end{align*}
where the middle term is controlled via
\begin{align*}
    \norm{Q_{\bx} - Q_{\by}}_{\rm TV}
    &\le \frac{\norm{\bx-\by}}{\sqrt{2h}}\,, \qquad\text{if}~h \le \frac{2}{\beta}\,,
\end{align*}
see~\cite[Lemma 3]{dwivedi2019log}.
Hence, we obtain:
\begin{align*}
    \frac{\sqrt{2h}}{6} &\leq \norm{\bx-\by}\,,
\end{align*}
which implies that $\operatorname{dist}(E \cap S_1, E\cap S_2) \geq \sqrt{2h}/6$. By the isoperimetric inequality (see Lemma~\ref{lem:properties_of_strong_log_concave}), there is an absolute constant $c > 0$ such that
\begin{align*}
    \pi\bigl({[(E\cap S_1)\cup (E \cap S_2)]}^\comp\bigr) \geq \frac{c \sqrt 2}{6} \sqrt{\alpha h} \,\pi(E \cap S_1)\,.
\end{align*}
Since $S_1$, $S_2$, and $S_3$ partition $\R^d$, we see that ${((E \cap S_1)\cup (E \cap S_2))}^\comp = E^\comp \cap S_3$. As a result,
\begin{align}
    \pi(S_3) + c_0 s \sqrt{\alpha h}
    &\ge \pi(S_3) + \pi(E^\comp)
    \ge \frac{c \sqrt 2}{6} \sqrt{\alpha h} \,\pi(E \cap S_1) \nonumber \\
    &\ge \frac{c \sqrt 2}{6} \sqrt{\alpha h} \, \{\pi(S_1) - \pi(E^\comp)\} \nonumber \\
    &\ge \frac{c \sqrt 2}{6} \sqrt{\alpha h} \, \bigl\{\frac{\pi(S)}{2} - \pi(E^\comp)\bigr\} \nonumber \\
    &\ge \frac{c \sqrt 2}{12} \sqrt{\alpha h} \, \pi(S) \label{eq:cond_arg} \,,
\end{align}
where~\eqref{eq:cond_arg} follows since $\pi(S)/2 \ge s/2 \ge 2c_0 s\sqrt h \ge 2\pi(E^\comp)$ provided that $c_0 \sqrt h \le 1/4$.

Since $\pi(S) \ge s$, it follows that, provided we choose $c_0$ small enough (and thus, the constant $c_1$ in the step size~\eqref{eq:cond_arg_step_size} small enough), we obtain
\begin{align*}
    \pi(S_3)
    &\ge \frac{c \sqrt 2}{24} \sqrt{\alpha h} \, \pi(S)\,.
\end{align*}
From this,
\begin{align*}
    \int_S T(\bx, S^\comp) \, \pi(\D \bx)
    &= \frac{1}{2} \Bigl(\int_S T(\bx, S^\comp) \, \pi(\D \bx) + \int_{S^\comp} T(\bx, S) \, \pi(\D \bx)\Bigr) \\
    &\ge \frac{1}{2} \cdot \frac{1}{4} \cdot \pi(S_3)
    \ge \frac{c\sqrt 2}{192} \sqrt{\alpha h} \, \pi(S)\,.
\end{align*}

Collecting the arguments, we obtain a lower bound on the $s$-conductance.

\begin{proposition}\label{prop:s_cond}
    If the step size $h$ is chosen as~\eqref{eq:cond_arg_step_size} for a sufficiently small constant $c_1$, then the $s$-conductance of the MALA chain satisfies
    \begin{align*}
        \msf C_s
        \gtrsim \sqrt{\alpha h} \,.
    \end{align*}
\end{proposition}

Together with the mixing time bound in Corollary~\ref{cor:mixing_from_s_conductance}, we have proven Theorem~\ref{thm:mala_upper_bd}.

\subsection{Auxiliary lemmas} \label{sec:pf_aux}

\subsubsection{Standard facts about strongly log-concave measures}\label{sec:pf_aux_strongly}
The following properties of strongly log-concave measures are well-known.

\begin{lemma}\label{lem:properties_of_strong_log_concave}
    The $\alpha$-strong convexity of $V$ implies the following properties:
    \begin{enumerate}
        \item (moment and tail bounds) For $\bx \sim \pi$, it holds that $\E\norm{\bx}^2 \le d/\alpha$.
        
        In fact, for all $k\ge 2$,
        \begin{align*}
            \E\norm{\bx}^k
            &\le \frac{3^k \, (d^{k/2} + k^{k/2})}{\alpha^{k/2}} \,.
        \end{align*}
        Consequently, $\E\exp(\lambda \, \norm{\bx}^2)$ is bounded above by a universal constant, provided that $0 \le \lambda \le \alpha/(40d)$.
        \item (isoperimetry) For any $S \subseteq \R^d$ with $\pi(A) \le 1/2$, it holds that $\pi(S^\varepsilon \setminus S) \gtrsim \varepsilon \sqrt{\alpha} \, \pi(S)$, where
        \begin{align*}
            S^\varepsilon := \{\bx \in \R^d \mid \exists y \in S~\text{with}~\norm{\bx-\by} \le \varepsilon\}.
        \end{align*}
        \item (sub-Gaussian concentration) For any $1$-Lipschitz function $f : \R^d\to\R$ and $\delta > 0$, with probability at least $1-\delta$ it holds that
        \begin{align*}
            f(\bx) - \E_\pi f \le \sqrt{\frac{2}{\alpha} \ln \frac{1}{\delta}},
        \end{align*}
        when $\bx \sim \pi$.
    \end{enumerate}
\end{lemma}
\begin{proof}
    The first statement is a simplification of~\cite[Lemma 2]{dalalyankaragulyanrioudurand2019nonstronglylogconcave}.
    For the second statement, in fact strongly log-concave measures satisfy a stronger isoperimetric inequality (sometimes called a Gaussian isoperimetric inequality, or a log-isoperimetric inequality in~\cite{chenetal2020hmc}); we refer to~\cite[\S 8.5.2]{bakrygentilledoux2014} and the paper~\cite{bobkovhoudre1997sobolevisoperimetry} which explains the relationship between  integral form of the isoperimetric inequality employed here and the more traditional differential version.
    Finally, for the third statement, see e.g.~\cite[\S 5.4.2, Corollary 5.7.2]{bakrygentilledoux2014}.
    
    Alternatively, these facts all follow from the corresponding facts about standard Gaussians, as a consequence of Caffarelli's contraction theorem~\citep{caffarelli2000contraction, fathigozlanprodhomme2020caffarelli}; see also the discussion in~\cite[\S 9.2.3]{villani2003topics}.
\end{proof}

\subsubsection{Stochastic calculus results}\label{sec:pf_aux_stochastic}
Below, we also collect together some inequalities proven via stochastic calculus.
In what follows, ${(\bar\bX_t)}_{t\ge 0}$ is the Langevin diffusion~\eqref{eq:langevin_sde}, started at $\bx$.
We start with a bound on the mean squared displacement $\E[\norm{\bar\bX_t - \bx}^2]$ of the Langevin diffusion.

\begin{lemma} \label{lem:ito_distance}
    If ${(\bar{\bX}_t)}_{t\ge 0}$ denotes the continuous-time Langevin process~\eqref{eq:langevin_sde} started at $\bx$, then for all $t\leq 1/(3\beta^{4/3})$, we have
    \begin{align*}
        \E[\norm{\bar{\bX}_t - \bx}^2]
        &\le 3 t\,(d+ \beta^{2/3} \, \norm{\bx}^2)\,.
    \end{align*}
\end{lemma}

\begin{proof}



    Fix $s\in[0,  t]$.
    From It\^o's lemma~\citep[Theorem 5.10]{legall2016stochasticcalc}, we have
\begin{align*}
    \E[\norm{\bar{\bX}_s- \bx}^2]
    &= \E\int_0^s \bigl\{-2\,\langle \nabla V(\bar{\bX}_u), \bar{\bX}_u -\bx \rangle + \frac{1}{2}\cdot 2 d\bigr\} \, \D u\\
    &= \E\int_0^s \{-2 \, \langle \nabla V(\bar{\bX}_u), \bar{\bX}_u - \bx \rangle\} \, 
    \D u + sd\,.
\end{align*}
To upper bound the first term on the right-hand side, we could conclude easily using a convexity of $V$ with slightly different dependence on $\beta$ in the final result. Instead, we take somewhat of a detour to show that this results hinges solely on the smoothness of $V$ and can therefore be extended beyond the log-concave case. 

Note that
\begin{align*}
    \abs{\langle \nabla V(\bar{\bX}_u), \bar{\bX}_u - \bx \rangle}
     &\le \abs{\langle \nabla V(\bar{\bX}_u)-\nabla V(\bx), \bar{\bX}_u -\bx \rangle} +\abs{\langle \nabla V(\bx), \bar{\bX}_u -\bx \rangle} \\
     &\le \beta \, \norm{\bar{\bX}_u -\bx}^2 + \frac{1}{2\beta^{4/3}}\, \norm{\nabla V(\bx)}^2+ \frac{\beta^{4/3}}{2} \, \norm{\bar{\bX}_u -\bx}^2\\
      &\le \frac{3\beta^{4/3}}{2}\, \norm{\bar{\bX}_u -\bx}^2 + \frac{\beta^{2/3}}{2} \, \norm{\bx}^2\,,
\end{align*}
where the last two inequalities follow from $\beta$-smoothness of $V$ (see e.g. \cite[Theorem 2.1.5]{nesterov2018lectures}), and our assumption $\argmin V=\bs 0$.
Thus, letting $a(u):= \E[\norm{\bar{\bX}_u- \bx}^2]$, we obtain  the following integral inequality:
\begin{align*}
  a(s) \leq (d+ \beta^{2/3}\, \norm{\bx}^2)\, s + 3\beta^{4/3} \int_0^s a(u) \,  \D u\,,\qquad \forall s\in[0,t]\,.
\end{align*}
Applying a version of Gr\"{o}nwall's inequality (e.g. \cite[Lemma 1.2.4]{stroock2018elements}), we obtain:
\begin{align*}
    a(t) &\leq t\, (d + \beta^{2/3} \,\norm{\bx}^2) \exp(3\beta^{4/3} t) 
    \leq 3t\,(d+ \beta^{2/3}\, \norm{\bx}^2)  \,,
\end{align*}
where the last line uses the hypothesis $t \leq 1/(3\beta^{4/3})$.
\end{proof}

In addition, we will also need a concentration inequality for $\norm{\bar\bX_t - \bx}^2$.
We first present a bound on the moment generating function of the supremum of a one-dimensional Brownian motion using the reflection principle.
\begin{lemma}\label{lem:Brownian_MGF}
    Let ${(B_s)}_{s\ge 0}$ be a standard one-dimensional Brownian motion. For $h,\lambda > 0$, such that $\lambda < \frac{1}{2h}$ the following holds:
    \begin{align*}
        \E \exp\Bigl( \lambda \sup_{s\in [0,h]} |B_s|^2 \Bigr) \le \frac{1+2h\lambda}{1-2h\lambda}.
    \end{align*}
\end{lemma}
\begin{proof}
    The reflection principle \cite[Proposition 6.19, 2.2.6]{karatzas1998brownian} states that for every $t > 0$,
    \begin{align*}
        \mbb P\bigl(\sup_{s\in[0,h]} B_s > t\bigr) = 2 \,\Pr(B_h > t).
    \end{align*}
    As a result, we have that
    \begin{align*}
        \mbb P\bigl(\sup_{s\in[0,h]} |B_s|^2 > t\bigr)
        &= \mbb P\bigl(\sup_{s\in[0,h]} |B_s| > \sqrt t\bigr)\\
        &\le \mbb P\bigl(\sup_{s\in[0,h]} B_s > \sqrt t\bigr) + \mbb P\bigl(\inf_{s\in[0,h]} B_s < - \sqrt t\bigr)\\
        &= 4 \, \Pr(B_h > \sqrt t)
        \le 2 \exp\bigl(- \frac{t}{2h}\bigr).
    \end{align*}
    Thus,
    \begin{align*}
        \E \exp\Bigl( \lambda \sup_{s\in [0,h]} |B_s|^2 \Bigr)
        &= 1+\lambda \int_0^\infty \exp(\lambda t) \, \mbb P\bigl(\sup_{s\in[0,h]} |B_s|^2 > t\bigr) \, \D t \\
        &\le 1+2\lambda \int_0^\infty \exp\bigl( - \frac{1-2h\lambda}{2h} \, t\bigr) \, \D t
        = 1+\frac{4h\lambda}{1-2h\lambda} \,.
    \end{align*}
\end{proof}

The above argument is relevant for Lemma~\ref{lem:sup_xt_control}, which is restated and proved below.
\lemsupcontrol*
\begin{proof}
    For a fixed realization of the sample path ${(\bar{\bX}_t)}_{t\in [0,h]}$ and $0 \le t \le h$, define the function $f(t) := \sup_{s\in [0,t]}{\norm{\bar\bX_s - \bx}^2}$.
    Then, for all $s \in [0,t]$,
    \begin{align*}
        \norm{\bar\bX_s - \bx}^2
        &= \Bigl\lVert -\int_0^s \nabla V(\bar\bX_r) \, \D r + \sqrt 2 \, \bB_s \Bigr\rVert^2
        \le 2 \, \Bigl\lVert -\int_0^s \nabla V(\bar\bX_r) \, \D r \Bigr\rVert^2 + 4 \, \norm{\bB_s}^2 \\
        &\le 2h \int_0^s \norm{\nabla V(\bar\bX_r)}^2 \, \D r + 4 \, \norm{\bB_s}^2
        \le 2\beta^2 h \int_0^s \norm{\bar\bX_r}^2 \, \D r + 4 \, \norm{\bB_s}^2 \\
        &\le 4\beta^2 h \int_0^s \norm{\bar\bX_r - \bx}^2 \, \D r + 4\beta^2 h^2 \, \norm{\bx}^2 + 4 \, \norm{\bB_s}^2 \\
        &\le 4\beta^2 h \int_0^s f(r) \, \D r + 4\beta^2 h^2 \, \norm{\bx}^2 + 4 \, \norm{\bB_s}^2
    \end{align*}
    which yields
    \begin{align*}
        f(t)
        &= \sup_{s \in [0,t]}{\norm{\bar\bX_s - \bx}^2}
        \le 4\beta^2 h \int_0^t f(r) \, \D r + 4\beta^2 h^2 \, \norm{\bx}^2 + 4 \sup_{s \in [0,h]}{\norm{\bB_s}^2}\,.
    \end{align*}
    Applying Gr\"{o}nwall's inequality~\citep[Lemma 1.2.4]{stroock2018elements}, we see that
    \begin{align*}
        f(h)
        &= \sup_{s \in [0,h]}{\norm{\bar\bX_s - \bx}^2}
        \le \bigl(4\beta^2 h^2 \, \norm{\bx}^2 + 4 \sup_{s \in [0,h]}{\norm{\bB_s}^2}\bigr) \exp(4\beta^2 h^2) \\
        &\le 12\beta^2 h^2 \, \norm{\bx}^2 + 12 \sup_{s \in [0,h]}{\norm{\bB_s}^2} \,.
    \end{align*}
    Hence,
    \begin{align*}
        \E \exp\bigl(\lambda \sup_{t\in [0,h]}{\norm{\bar\bX_t - \bx}^2}\bigr)
        &\le \exp(12\beta^2 \lambda h^2 \, \norm{\bx}^2) \E\exp\bigl(12\lambda \sup_{s \in [0,h]}{\norm{\bB_s}^2}\bigr) \\
        &\le \exp(12\beta^2 \lambda h^2 \, \norm{\bx}^2) \, {\bigl\{ \E\exp\bigl(12\lambda \sup_{s \in [0,h]}{\abs{B_s}^2}\bigr)\bigr\}}^d \\
        &\le \exp(12\beta^2 \lambda h^2 \, \norm{\bx}^2) \, \bigl( \frac{1+24h\lambda}{1-24h\lambda} \bigr)^d\,,
    \end{align*}
    by Lemma~\ref{lem:Brownian_MGF} and the assumption $\lambda < 1/(24h)$.
\end{proof}

\begin{corollary}\label{cor:moments_of_langevin}
    Assume $h \le 1/(2\beta)$.
    There exists a numerical constant $C > 0$ such that for all $k\ge 1$,
    \begin{align*}
        \E \sup_{t\in [0,h]}{\norm{\bar \bX_t - \bx}^{2k}}
        &\le C^k \, (\beta^k h^{2k} \, \norm{\bx}^{2k} + d^k h^k + h^k k^k).
    \end{align*}
\end{corollary}
\begin{proof}
    In Lemma~\ref{lem:sup_xt_control}, take $\lambda := 1/(48h)$ to yield
    \begin{align*}
        \E\exp\bigl(\lambda \sup_{t\in [0,h]}{\norm{\bar\bX_t - \bx}^2}\bigr)
        &\le \exp\bigl(\frac{1}{4}\beta^2 h \, \norm{\bx}^2 + d\ln 3\bigr) \,.
    \end{align*}
    It follows from Markov's inequality that for all $x \ge 0$,
    \begin{align*}
        \Pr\bigl(\sup_{t\in [0,h]}{\norm{\bar\bX_t - \bx}^2} \ge 12h^2 \beta \, \norm{\bx}^2 + (48 \ln 3) hd + x\bigr)
        &\le \exp\bigl( - \frac{x}{48h}\bigr).
    \end{align*}
    The result now follows from standard moment bounds under sub-exponential concentration~\citep[see, e.g.,][Proposition 2.7.1]{vershynin2018highdimprob}.
\end{proof}

\begin{remark}
    Bounds such as the one in Corollary~\ref{cor:moments_of_langevin} are standard and have appeared in the literature before, e.g.,~\cite[Lemma 11]{mouetal2019improvedlangevin}.
\end{remark}

\subsection{From total variation to other distances}\label{scn:other_distances}

In this section, we deduce the mixing time results of Theorem~\ref{thm:mala_upper_bd} for the KL divergence, the chi-squared divergence, and the $2$-Wasserstein distance.

We begin with the following lemma which shows that the warmness parameter (defined in Definition~\ref{def:warmstart}) is preserved by the iterations of MALA. In fact, this is true for all reversible Markov chains.

\begin{lemma}\label{lem:getting_hotter_by_the_minute}
    Let ${(\mu_n)}_{n\in\N}$ denote the iterates of a Markov chain whose kernel $T$ is reversible with respect ot $\pi$, and assume that $\mu_0$ is $M_0$-warm with respect to $\pi$.
    Then, for all $n\in\N$, the iterate $\mu_n$ is also $M_0$-warm with respect to $\pi$.
\end{lemma}
\begin{proof}
    The proof is by induction. For any $\by \in \R^d$,
    \begin{align*}
        \frac{\mu_{n+1}(\by)}{\pi(\by)}
        &= \int \frac{\mu_n(\bx)}{\pi(\by)} \, T(\bx, \by) \, \D \bx
        = \int \frac{\mu_n(\bx)}{\pi(\bx)} \, \frac{\pi(\bx) T(\bx, \by)}{\pi(\by)} \, \D \bx
        \le M_0 \int T(\by, \bx) \, \D \bx
        = M_0\,,
    \end{align*}
    where we use the inductive assumption and the reversibility of $T$.
\end{proof}

Under a warmness condition, the total variation distance controls the chi-squared divergence.

\begin{lemma}\label{lem:chi_sq_control_under_warm}
    Let $\mu$ be $M_0$-warm with respect to $\pi$.
    Then,
    \begin{align*}
        \chi^2(\mu \mmid \pi)
        &\le 2M_0 \, \norm{\mu-\pi}_{\rm TV}\,.
    \end{align*}
\end{lemma}
\begin{proof}
    From the definition of the chi-squared divergence,
    \begin{align*}
        \chi^2(\mu \mmid \pi)
        &= \int \bigl\lvert \frac{\mu}{\pi} - 1 \bigr\rvert^2 \, \D \pi
        \le M_0 \int\bigl\lvert \frac{\mu}{\pi} - 1 \bigr\rvert \, \D \pi
        = 2M_0 \, \norm{\mu - \pi}_{\TV}\,.
    \end{align*}
    Here we use the fact that pointwise, $\abs{\mu/\pi-1} \le \max\{1, M_0 - 1\} \le M_0$.
\end{proof}

It immediately implies the following result on mixing times.

\begin{corollary}\label{cor:mixing_other_distances}
    Fix $\varepsilon > 0$.
    Then, MALA initialized with a distribution $\mu_0$ which is $M_0$-warm with respect to $\pi$ satisfies the following mixing time bounds:
    \begin{align*}
        \tau_{\rm mix}(\varepsilon, \mu_0; \msf d)
        &\le \tau_{\rm mix}\bigl( \frac{\varepsilon^2}{2M_0}, \mu_0; \TV \bigr)
    \end{align*}
    for each of the distances
    \begin{align*}
        \msf d \in \bigl\{\sqrt{\KL}, \; \sqrt{\chi^2}, \; \sqrt{\frac{\alpha}{2}} \, W_2\bigr\}\,.
    \end{align*}
\end{corollary}
\begin{proof}
    The mixing time in the chi-squared distance is a straightforward consequence of Lemmas~\ref{lem:getting_hotter_by_the_minute} and~\ref{lem:chi_sq_control_under_warm}.
    The result for the KL divergence now follows since $\KL \le \chi^2$~\citep[Lemma 2.7]{tsybakov2009nonparametric}.
    Finally, for the result in $2$-Wasserstein distance we can use Talagrand's transportation inequality
    \begin{align*}
        \frac{\alpha}{2}\, W_2^2(\mu, \pi)
        &\le \KL(\mu \mmid \pi), \qquad \text{for all probability measures}~\mu \ll \pi\,,
    \end{align*}
    which is a consequence of the strong convexity of $V$~\citep[in fact it is a consequence of the weaker assumption of a log-Sobolev inequality, see][Theorem 9.6.1]{bakrygentilledoux2014}.
\end{proof}

Corollary~\ref{cor:mixing_other_distances} implies the remaining mixing time results in Theorem~\ref{thm:mala_upper_bd}.

\section{Proof of the lower bound}\label{scn:lower_bd_proof}

This section presents the proofs of Theorems~\ref{thm:mala_lower_bd_main} and \ref{thm:spectral_gap_upper_bd}.
The majority of this section is devoted to the proof of the upper bound on the conductance when $h \gg d^{-1/2}$ (Theorem~\ref{thm:mala_lower_bd_main}).
The proof of the upper bound on the spectral gap (Theorem~\ref{thm:spectral_gap_upper_bd}) is given in Appendix~\ref{scn:spectral_gap_upper_bd}.

\subsection{High-level overview of the proof}

Recall that we take $\ee = 1/4 - \delta$, where $\delta > 0$ is fixed throughout.
As mentioned in Section~\ref{scn:lower_bd}, we consider the potential
\begin{align}\label{eq:V_lower_bd}
    V(\bx)
    &= \frac{\|\bx\|^2}{2}- \frac{1}{2d^{2\ee}} \sum_{i=1}^d \cos(d^\ee x_i) \\
    &=: \fo(\bx) + \ft(\bx).
\end{align} 
From the construction, it immediately follows that $V$ is  $1/2$-strongly convex and $3/2$-smooth.

We begin with some intuition for the above construction. 
At a high level, our construction can be seen as a ``perturbed'' Gaussian distribution; $\fo$ is the potential corresponding to a standard Gaussian and $\ft$ corresponds to a perturbation.
Having this interpretation, we are interested in constructing a distribution (i) that is significantly different from the standard Gaussian, yet (ii) the difference is not noticed by each step of MALA. \begin{enumerate}
    \item[(i)] A quick calculation (see Lemma~\ref{lem:kl_of_adversarial}) shows that $\KL(\mc N(0,1) \mmid \pi) = O(d^{1-4\ee})$. 
    So, we must take $\ee \le 1/4$ to ensure that $\pi$ is significantly different from the standard Gaussian.
    \item[(ii)] On the other hand,   $\ft$ is an oscillatory perturbation.
    Hence, MALA would not see the contribution from $\ft$ as long as its movement due to the Langevin proposal is at least as long as the length scale of the fluctuations of $\ft$.

    With this in mind, note that  the fluctuations of $\ft$ is of order $d^{-\ee}$, while the movement of a single coordinate under the Langevin proposal is of order $\sqrt h$ (due to the Gaussian part).
    Hence, MALA would essentially ignore $\ft$ as long as $h \gg d^{-2\ee}$.
\end{enumerate}
We formalize the above heuristic in the rest of this section.




To prove the upper bound on the conductance in Theorem~\ref{thm:mala_lower_bd_main}, we use the following proposition.

\begin{proposition} \label{prop:upper_conductance}
Let $\eve$ be an event such that $
\pi(\eve)\geq 1/2$. Then,
\begin{align*}
    \msf C \leq 2\sup_{\bx \in \eve} \int_{\R^d} Q(\bx, \by) A(\bx, \by) \, \D \by\,.
\end{align*}
\end{proposition}
\begin{proof}
Let $\eve_0$ be a subset of $\eve$ with $\pi(\eve_0)=1/2$.
From the definition of the conductance~\eqref{eq:C}, 
\begin{align*}
    \msf C &=\inf_{\substack{S \subseteq \R^d \\ \pi(S) \le 1/2}} \frac{\int_S T(\bx, S^\comp) \, \pi(\D \bx)}{\pi(S)} \leq 2\int_{\eve_0}T(\bx, \eve_0^\comp) \, \pi(\D \bx)\\
    &\leq 2\int_{\eve_0} \Bigl( \int_{\eve_0^\comp} Q(\bx, \by) A(\bx, \by) \, \D \by \Bigr) \, \pi(\bx) \, \D \bx 
    \le  2\int_{\eve_0} \Bigl( \int_{\R^d} Q(\bx, \by) A(\bx, \by) \, \D \by \Bigr) \, \pi(\bx) \, \D \bx\\
    &    \le 2 \sup_{\bx \in \eve_0} \int_{\R^d} Q(\bx, \by) A(\bx, \by) \, \D \by \leq  2\sup_{\bx \in \eve} \int_{\R^d} Q(\bx, \by) A(\bx, \by) \, \D \by. 
\end{align*}    
\end{proof}
From Proposition~\ref{prop:upper_conductance}, it therefore suffices to show that there is an event $\eve \subseteq \R^d$ with probability $\pi(\eve)\geq 1/2$ such that
\begin{align*}
     \sup_{\bx \in \eve} \int_{\R^d} Q(\bx, \by) A(\bx, \by) \, \D \by \le \exp[-\Omega(d^{4\delta})]\, .
\end{align*} 
By definition of the Metropolis-Hasting accept-reject step~\eqref{eq:accept_prob}, we have
\begin{align}
    Q(\bx,\by) A(\bx,\by) 
    &= Q(\bx,\by)\min\Bigl\{ 1,  \frac{\pi(\by) Q(\by, \bx)}{\pi(\bx) Q(\bx,\by)}\Bigr\} \nonumber \\
    &\le \frac{\pi(\by) Q(\by, \bx)}{\pi(\bx) } \nonumber \\
    &= \frac{1}{{(4\uppi h)}^{d/2}} \exp\Bigl[ V(\bx) - V(\by) - \frac{\norm{\by-\bx-h \nabla V(\by)}^2}{4h} \Bigr]\,. \label{eq:acc_prop_upper}
\end{align} 
We substitute in the definition of our potential~\eqref{eq:V_lower_bd} and expand out the terms in~\eqref{eq:acc_prop_upper}, grouping them according to whether they involve $\ft$ or not:
\begin{align} 
\eqref{eq:acc_prop_upper} &=\frac{1}{{(4\uppi h)}^{d/2}} \exp\Bigl[ \frac{1}{2}\norm{\bx}^2 - \frac{1}{2}\norm{\by}^2 - \frac{1}{4h} \norm{(1-h)\by-\bx }^2 \Bigr] \label{exp:acc1}\\
&\qquad{} \times \exp\Bigl[ \ft(\bx) - \ft(\by) + \frac{1}{2}\inp{(1-h)\by-\bx}{\nabla \ft(\by)} -  \frac{h}{4}\, \norm{\nabla \ft(\by)}^2 \Bigr]\,. \label{exp:acc2} 
\end{align} 
Some algebra yields that \eqref{exp:acc1} is equal to 
\begin{align*}
    \underbrace{\bigl(\frac{1+h^2}{4\uppi h}\bigr)^{d/2} \exp\Bigl[ -\frac{1+h^2}{4h} \, \bigl\lVert \by-\frac{1-h}{1+h^2} \,\bx\bigr\rVert^2 \Bigr]}_{=:\mux(\by)} \, \frac{1}{{(1+h^2)}^{d/2}}\exp\Bigl[  \frac{h^2\, \norm{\bx}^2}{2\,(1+h^2)} \Bigr]\,.
\end{align*} 
The first term, which we denote by $\mux(\by)$, is the probability density function of the distribution $\Nor (\frac{1-h}{1+h^2}\,\bx, \frac{2h}{1+h^2} I_d)$ evaluated at $\by$.
Using this observation, the quantity $\int_{\R^d} Q(\bx,\by) A(\bx,\by)\, \D \by$ is upper bounded by
\begin{align*}
     \underbrace{\frac{\exp\Bigl[  \frac{h^2 \,\norm{\bx}^2}{2\,(1+h^2)} +\ft(\bx)  \Bigr]}{(1+h^2)^{d/2}}}_{\circled{1}}   \times\underbrace{\Ex_{\by\sim \mux}\exp\Bigl[   - \ft(\by) + \frac{1}{2}\inp{(1-h)\by-\bx}{\nabla \ft(\by)} -  \frac{h}{4}\, \norm{\nabla \ft(\by)}^2   \Bigr]}_{\circled{2}}\,.
\end{align*}

Having this upper bound, we will prove that
there is a set $E \subseteq \R^d$ with $\pi(E) \ge 1/2$ such that the following bounds hold for all $\bx \in E$:
\begin{enumerate}
    \item (Lemma~\ref{lem:control_1})
    $$\tcone \leq \exp\bigl[-\frac{1}{8}d^{1-4\ee} + o(d^{1-4\ee})\bigr] \,.$$ 
    \item (Lemma~\ref{lem:control_2})
    $$\tctwo \leq \exp\bigl[\frac{1}{16} d^{1-4\ee} + o(d^{1-4\ee})\bigr]\,.$$
\end{enumerate}
From these bounds and the preceding calculations, we have
\begin{align*}
    \sup_{\bx \in E} \int Q(\bx,\by) A(\bx,\by) \, \D \by
    &\le \exp\bigl[-\frac{1}{8}d^{1-4\ee} + o(d^{1-4\ee})\bigr] \,.
\end{align*}
This completes the proof of Theorem~\ref{thm:mala_lower_bd_main}.

The next section is devoted to proving the two main bounds (Lemmas~\ref{lem:control_1} and \ref{lem:control_2}).

\subsection{Proofs of technical statements}

\subsubsection{Notation and technical lemmas}
We use the following notation:
\begin{align}\label{def:pi_bad}
\begin{cases}
     \vb(x):= \frac{1}{2}x^2 -\frac{1}{2} d^{-2\ee}\cos(d^\ee x)\,, \\
    \Vb(\bx):= \sum_{i=1}^d \vb(x_i) = \frac{1}{2}\norm{\bx}^2 -\frac{1}{2}d^{-2\ee}\sum_{i=1}^d \cos(d^\ee x_i)\,,\\
    \pib(x)\propto \exp(-\vb(x) ) \,,\\
     \Pib(\bx)\propto \exp(-\Vb(\bx) )\,.
\end{cases}
\end{align}
Thus, $\pi_1$ is the marginal distribution of $\pi$.
We first list useful technical lemmas for proving Lemmas~\ref{lem:control_1} and \ref{lem:control_2}.
First, the following trigonometric inequality will be used several times.
\begin{lemma}\label{lem:trig}
    Let $\xi \sim \Nor (0,1)$, let $p$ be a polynomial, and let $a, b \in \R$, $\cc > 0$ be constants. Then, there exists a constant $C$ (depending on $p$, $a$, $b$, and $\gamma$) such that
    \begin{align*}
        \abs{\E[p(\xi) \sin(a + bd^\cc \xi)]} \le \frac{C}{d}\,.
    \end{align*}
\end{lemma}
\begin{proof}
The key fact we use  is that the characteristic function $\E[e^{it\xi}]$ of a Gaussian is equal to $e^{-\frac{1}{2}t^2}$. 
First consider the case $p \equiv 1$.
Let $\Im(\cdot)$ denote the imaginary part.
Then, we have
\begin{align*}
    \E[\sin(a + b d^\cc \xi)]
    &= \E[\Im(e^{i\, (a + bd^\cc \xi)})]\\
    &= \Im(e^{ia} \E[e^{i bd^\cc \xi}])\\
    &=\Im\Bigl(\exp\bigl(ia - \frac{b^2 d^{2\cc}}{2}\bigr)\Bigr)\\
    &= \sin(a) \exp\bigl(-\frac{b^2 d^{2\cc}}{2}\bigr)\,.
\end{align*}
It is then clear that the result holds for $p = 1$.
Next, when $p(x)=x^\ell$ for some $\ell \in\N^+$,
\begin{align*}
    \E[\xi^\ell \sin(a + b d^\cc \xi)]
    &= \Im(e^{ia}\E[\xi^\ell e^{i b d^\cc \xi}])\\
    &= \Im\Bigl(e^{ia} \, i^{-\ell} \E\Bigl[\frac{\D^\ell}{\D t^\ell} e^{it \xi}\Big|_{t = b d^\cc} \Bigr]\Bigr)\\
    &= \Im\Bigl(e^{ia} \, i^{-\ell} \, \frac{\D^\ell}{\D t^\ell} e^{-\frac{t^2}{2}}\Big|_{t = bd^\cc}\Bigr).
\end{align*}
Thus, it is clear that the lemma holds for this choice of $p$ too.
The case of a general polynomial follows from linearity.
\end{proof}

Clearly, the statement of the previous lemma can be substantially strengthened, but this will not be necessary for the MALA lower bound.

Now we list some useful facts about the adversarial target  distribution. 

\begin{lemma}\label{lem:pibad}
Assume $\ee < 1/4$.
The following hold for $\pib$ and $\Pib$ defined in \eqref{def:pi_bad}:
\begin{enumerate}[label= (\alph*)]
    \item Let $\zb:= \int_{\R} \exp(-\vb(\bx)) \, \D x$ be the one-dimensional normalizing constant. Then, we have $\zb =\sqrt{2\uppi} +O(d^{-4\ee})$. 
    \item  $ \E_{x\sim \pib} [x^2] \leq 1+ O(d^{-4\ee})$.  Consequently, $\E_{\bx\sim \Pib} [\norm{\bx}^2] \leq d+ O(d^{1-4\ee})$.
    \item $ \E_{x\sim \pib} [\cos(d^\ee x)] \leq \frac{1}{4}d^{-2\ee}+ O(d^{-6\ee})$.
\end{enumerate}
\end{lemma} 
\begin{proof}
    \begin{itemize}
        \item[(a)] Letting $\xi\sim \Nor (0,1)$, then
    \begin{align*}
        \zb - \sqrt{2\uppi} &=   \int_{\R} \exp\Bigl(-\frac{1}{2} \, x^2 +\frac{1}{2d^{2\ee}} \cos(d^\ee x)\Bigr) \, \D x-\sqrt{2\uppi}\\
        &= \sqrt{2\uppi} \int_{\R}\exp\Bigl(\frac{1}{2d^{2\ee}} \cos(d^\ee x)\Bigr) \, \frac{\exp(-\frac{1}{2}x^2 )}{\sqrt{2\uppi}}\, \D x - \sqrt{2\uppi}\\
        &= \sqrt{2\uppi}\,\Bigl(\E \exp\bigl( \frac{1}{2d^{2\ee}}\, \cos(d^\ee \xi)\bigr) - 1\Bigr)\\
        &= \frac{\sqrt{2\uppi}}{2d^{2\ee}} \E\cos(d^\ee \xi) + O(d^{-4\ee}).
    \end{align*}
By Lemma~\ref{lem:trig}, we have $\abs{\E\cos(d^\ee \xi)} = O(d^{-1}) = o(d^{-4\ee})$, since $\ee < 1/4$. The proof of (a) then follows.
\item[($b$)] Similarly, letting $\xi\sim \Nor (0,1)$,
    \begin{align*}
       \E_{x\sim \pib} [x^2] &= \int x^2 \, \frac{\exp(-\vb(\bx))}{\zb} \, \D x\\
        &= \frac{\sqrt{2\uppi}}{\zb} \E\bigl[\xi^2 \exp\bigl( \frac{1}{2d^{2\ee}} \cos(d^\ee \xi)\bigr)\bigr]\\
         &= \bigl(1+O(d^{-4\ee})\bigr) \E\bigl[\xi^2 \exp\bigl( \frac{1}{2d^{2\ee}} \cos(d^\ee \xi)\bigr)\bigr].
    \end{align*}
    By Taylor expansion,
    \begin{align*}
        \E\bigl[\xi^2 \exp\bigl( \frac{1}{2d^{2\ee}} \cos(d^\ee \xi)\bigr)\bigr]
        &= 1 + \frac{1}{2d^{2\ee}} \E[\xi^2 \cos(d^\ee \xi)] + O( d^{-4\ee}).
    \end{align*}
   Again by Lemma~\ref{lem:trig}, the second term is $O(d^{-(2\ee+1)}) = o(d^{-6\ee})$. Hence, the result follows.
\item[($c$)]  Similarly, it holds that 
    \begin{align*}
        \E_{x\sim \pib}\cos(d^\ee x)
        &= \frac{\sqrt{2\uppi}}{Z} \E\bigl[\cos(d^\ee \xi) \exp\bigl( \frac{1}{2d^{2\ee}} \cos(d^\ee \xi)\bigr)\bigr] \\
        &= \bigl(1 + O( d^{-4\ee})\bigr) \, \Bigl[ \E\cos(d^\ee \xi) + \frac{1}{2d^{2\ee}} \E\cos^2(d^\ee \xi) + O( d^{-4\ee})\Bigr].
    \end{align*}
    By Lemma~\ref{lem:trig}, the first term is $\E\cos(d^\ee \xi) = o(d^{-4\ee})$. Next, the second term is
    \begin{align*}
        \frac{1}{2d^{2\ee}} \E\cos^2(d^\ee \xi)
        &= \frac{1}{4d^{2\ee}} + \frac{1}{4d^{2\ee}} \E\cos(2d^\ee \xi).
    \end{align*}
    From Lemma~\ref{lem:trig}, $\E\cos(2d^\ee \xi) = o(d^{-4\ee})$. Therefore, the result follows.
    \end{itemize} 
\end{proof}

\begin{lemma}\label{lem:supcontrol}
For $\bx\sim \Pib$, the following holds with probability at least $1-1/(4d)$:
\begin{align*}
    \norm{\bx}_\infty < 4\sqrt{\ln(8d)}.
\end{align*}
\end{lemma}
\begin{proof}
    By symmetry, we just need to show that with probability at least $1-1/(8d)$,
    \begin{align*}
        \max_{i\in [d]} x_i < 4\sqrt{\ln d}\,.
    \end{align*}
    Since $V_1'' \geq 1/2$, each $|x_i|$ will be stochastically dominated by $|\xi|$, where $\xi\sim \Nor(0,2)$. Hence, if $\xi_1,\dotsc,\xi_d$ are i.i.d.\ copies of $\xi$, we just need to show that
    \begin{align*}
        \max_{i\in [d]} \xi_i < 4 \sqrt{\ln d}
    \end{align*}
    with probability at least $1-1/d$. The standard argument based on the moment generating function (e.g.~\cite[Lemma 5.1]{van2014probability}) tells us that $\E[\max_{i\in [d]} \xi_i] \leq 2\sqrt{\ln d}$, and Gaussian concentration (e.g.~\cite[Theorem 3.25]{van2014probability}) implies
    \begin{align*}
        \Pr\bigl(\max_{i\in [d]} \xi_i > \E \max_{i \in [d]}  \xi_i + t\bigr) \leq \exp\bigl(-\frac{t^2}{4}\bigr).
    \end{align*}
    Plug in $t = 2\sqrt{\ln(8d)}$ and we get the lemma as claimed.
\end{proof}

Now let us state and prove the technical statements in order.

\subsubsection{Proof of Lemma~\ref{lem:control_1}}

\begin{lemma}\label{lem:control_1} 
Assume that $0 < h \le d^{-1/3}$. Then there exists an event $\eve_1$ with $\Pib(\eve_1)\geq 3/4$ such that for $\bx \in \eve_1$,
    \begin{align*}
     \frac{\exp\Bigl[  \frac{h^2 \, \norm{\bx}^2}{2 \, (1+h^2)} +\ft(\bx)  \Bigr]}{{(1+h^2)}^{d/2}}\leq   \exp\bigl[-\frac{1}{8}d^{1-4\ee} + o(d^{1-4\eta})\bigr]\,.
    \end{align*}
\end{lemma}

\begin{proof} We decompose the left-hand side as
\begin{align*}
     \frac{\exp\Bigl[  \frac{h^2 \, \norm{\bx}^2}{2\, (1+h^2)} +\ft(\bx)  \Bigr]}{{(1+h^2)}^{d/2}} = \frac{1}{{(1+h^2)}^{d/2}}\exp\Bigl[  \frac{h^2 \,\norm{\bx}^2}{2 \, (1+h^2)}   \Bigr] \times \exp[\ft(\bx)]
\end{align*}
and bound each term separately. 

We begin with the first term.
By Lemma~\ref{lem:pibad}-(b), we know that  the second moment of $\Pib$ is $d + O(d^{1-4\ee})$.
Since $\pi$ is $1/2$-strongly log concave, a standard concentration argument (see e.g.\ Lemma~\ref{lem:properties_of_strong_log_concave}) shows that there exists a subset $\eve_1'$ with $\pi(\eve_1')\geq 7/8$  such that for $\bx\in \eve_1'$,
\begin{align*}
    \normb{ \bx}^2 \le \,d + O(d^{1-4\ee}) + O(d^{1/2})\,.
\end{align*}
Now, using the fact that $\ln(1+x)\ge x-x^2/2$ for $x\geq 0$,
\begin{align*}
    \frac{1}{{(1+h^2)}^{d/2}} \exp\Bigl[ \frac{h^2 \, \norm \bx^2}{2 \, (1+h^2)} \Bigr]
    &\le \exp\Bigl[ \frac{h^2 \, (d+O(d^{1-4\ee}) + O(d^{1/2}))}{2\, (1+h^2)} - \frac{d}{2} \ln(1+h^2)\Bigr] \\
    &\le \exp\Bigl[ \frac{h^2 \, (d+O(d^{1-4\ee}) + O(d^{1/2}))}{2\, (1+h^2)} - \frac{dh^2}{2}+ \frac{dh^4}{4}\Bigr] \\
      &= \exp\Bigl[ \frac{h^2 \, (O(d^{1-4\ee}) + O(d^{1/2}))}{2\, (1+h^2)} - \frac{dh^4}{2\, (1+h^2)}+  \frac{dh^4}{4}\Bigr] \\
        &= \exp\Bigl[ \frac{h^2 \, (O(d^{1-4\ee}) + O(d^{1/2}))}{2\, (1+h^2)} + \frac{-dh^4 + 2dh^6}{4\, (1+h^2)} \Bigr] \\
    &\le \exp[ O(d^{1-4\ee}h^2) + O(d^{1/2} h^2) ]  \,.
\end{align*}
where the last line follows since $h^2 \leq 1/2$. 
In order to show that the exponent of the above term is $o(d^{1-4\ee})$, we must check that $d^{1/2} h^2 = o(d^{1-4\eta})$, which holds if $h = o(d^{1/4-2\ee}) = o(d^{-1/4 + 2\delta})$. This indeed follows from our assumption that $h \le d^{-1/3}$.



Next, we move on to the second term. Recall from the calculation in Lemma~\ref{lem:pibad}-(c) that
 $ \E_{x\sim \pib} [\cos(d^\ee x)] \leq \frac{1}{4}d^{-2\ee}+ O(d^{-6\ee})$.
 Hence,
 it follows that
 \begin{align*}
     \Ex_{\bx\sim\pi} [\ft(\bx)] = -\frac{1}{2d^{2\ee}} \sum_{i=1}^d \E_{x_i \sim \pib} \cos(d^\ee x_i) = -\frac{1}{8}d^{1-4\ee}+ O(d^{1-8\ee}).
 \end{align*}
Since $\pi$ is $1/2$-strongly log-concave, another sub-Gaussian concentration argument (Lemma~\ref{lem:properties_of_strong_log_concave}) shows that there exists a subset $\eve_1''$ with $\pi(\eve_1'')\geq 7/8$  such that for $\bx \in \eve_1''$, 
\begin{align*}
    \exp[\ft(\bx)]\leq \exp\bigl[-\frac{1}{8}d^{1-4\ee} +O(d^{1-8\ee}) + O(d^{1/2-2\eta})\bigr] \leq \exp\bigl[-\frac{1}{8}d^{1-4\ee} +o(d^{1-4\ee}  )\bigr]\,,
\end{align*} 
since $1-4\eta > 0$ by the hypothesis.

Now taking $\eve_1 :=\eve_1'\cap \eve_1''$, the above calculations show that for $\bx\in \eve_1$, 
\begin{align*}
    \frac{\exp\Bigl[  \frac{h^2 \,\norm{\bx}^2}{2\,(1+h^2)} +\ft(\bx)  \Bigr]}{{(1+h^2)}^{d/2}}&\le \exp\bigl[-\frac{1}{8}d^{1-4\ee} +o(d^{1-4\ee} )\bigr]\,,
\end{align*}
which completes the proof.
\end{proof}

\subsubsection{Proof of Lemma~\ref{lem:control_2}}

\begin{lemma}\label{lem:control_2}
Assume that $h \in [  d^{-\frac1 2 + 3\delta}, d^{-\frac{1}{3}}]$.
Then there exists an event $\eve_2$ with $\Pib(\eve_2)\geq 3/4$ such that for $\bx \in \eve_2$,
     \begin{align*}
        &\Ex_{\by\sim \mux}\exp\Bigl[   - \ft(\by) + \frac{1}{2}\inp{(1-h)\by-\bx}{\nabla \ft(\by)} -  \frac{h}{4}\, \norm{\nabla \ft(\by)}^2   \Bigr] \\
        &\qquad\qquad\qquad\qquad\qquad\qquad\qquad\qquad\qquad\qquad \leq \exp\bigl[\frac{1}{16}d^{1-4\ee}  +o(d^{1-4\ee}) \bigr]\,.
    \end{align*}
\end{lemma}

\begin{proof} 
Recall the definition $\ft(\bx) = - \frac{1}{2}d^{-2\ee} \sum_{i=1}^d \cos(d^\ee x_i)$.
Since $\ft$ is separable, it suffices to consider the following quantity: for $\mu_{x_i}:=\Nor(\frac{1-h}{1+h^2} \, x_i,\frac{2h}{1+h^2})$,
\begin{align} \label{exp:1d}
    \max_{i\in [d]} \Ex_{y_i\sim \mu_{x_i}}\exp \Bigl(  \frac{\cos(d^\ee y_i)}{2d^{2\ee}}   + \frac{((1-h)y_i - x_i ) \sin(d^\ee y_i) }{4d^\ee}   - \frac{h\sin^2(d^\ee y_i) }{16 d^{2\ee}}  \Bigr) \,.
\end{align}
Indeed, the lemma is proved as soon as we show
\begin{align}\label{ineq:goal}
    \eqref{exp:1d} \leq \exp\bigl[\frac{1}{16}d^{-4\ee}  +o(d^{-4\ee}) \bigr]\,.
\end{align}
For the proof, we will therefore work with a single coordinate; for simplicity of notation, we will use the first coordinate.

To prove the inequality \eqref{ineq:goal}, let us first simplify the expression \eqref{exp:1d}.
Letting $\xi \sim \Nor(0,1)$, we can equivalently write $y_1= \frac{1-h}{1+h^2} \, x_1 +\sqrt{\frac{2h}{1+h^2}}\, \xi$.
From this, we get
\begin{align*}
    (1-h)y_1-x_1 = -\frac{2h}{1+h^2} \, x_1+(1-h)\sqrt{\frac{2h}{1+h^2}}\, \xi.
\end{align*}
Since our regime of interest is $h=o(1)$, we simplify the notation by defining
\begin{align*}
\hho:= \frac{h}{1+h^2} \qquad \text{and}\qquad \hht := \frac{{(1-h)}^2}{1+h^2}\,h\,,
\end{align*}
and treat them as being on the same order as $h$.
Using these simplifying notations and rearranging, we are left to consider 
\begin{align} \label{exp:1d_2}
         \Ex \exp \Bigl( \underbrace{\frac{\cos(d^\ee y_1)}{2d^{2\ee}}}_{=:\aaa}  - \underbrace{\frac{h\sin^2(d^\ee y_1) }{16 d^{2\ee}}}_{=:\bbb}   - \underbrace{\frac{2\hho x_1 \sin(d^\ee y_1) }{4d^\ee}}_{=:\ccc} +\underbrace{\frac{\sqrt{2\hht}\xi \sin(d^\ee y_1) }{4d^\ee}}_{=:\ddd}   \Bigr) \,,
\end{align}
where $y_1= \frac{1-h}{1+h^2} \, x_1 +\sqrt{\frac{2h}{1+h^2}}\, \xi$.
Now we will estimate \eqref{exp:1d_2} by a Taylor expansion.

Throughout, we will assume $\norm{\bx}_\infty \le 4\sqrt{\ln(8d)}$. By Lemma~\ref{lem:supcontrol}, this holds on an event $E_2$ of probability $\pi(E_2) \ge 3/4$.
From this, we note the immediate bounds
\begin{align*}
    |\aaa|= O(d^{-2\ee}), \qquad |\bbb| =O(d^{-2\ee}h), \qquad |\ccc| = \widetilde O(d^{-\ee }h), \qquad |\ddd|=O_{\msf p}(d^{-\ee}\sqrt{h}).
\end{align*}
Here, $O_{\msf p}$ denotes probabilistic big-O notation. Using  $h=O(d^{-1/3}) =o(d^{-4\ee/3})$, we have
\begin{align}\label{eq:order_of_terms}
    |\aaa|= O(d^{-2\ee}), \quad |\bbb| =o(d^{-(3+1/3)\ee}), \quad |\ccc| =  o(d^{-(2+1/3)\ee}), \quad |\ddd|=o_{\msf p}(d^{-(1+2/3)\ee}).
\end{align}
From, this, we see that the third- or higher-order terms in the Taylor expansion, after taking the expectation, are $o(d^{-5\ee})$. Indeed, the dominant term is $\E[\abs{\ddd}^3] = o(d^{-5\ee})$.

We also note that the common argument of the trigonometric terms is
\begin{align*}
    d^\eta y_1
    &= d^\eta \, \frac{1-h}{1+h^2} \, x_1 + d^\eta \sqrt{\frac{2h}{1+h^2}} \, \xi \,,
\end{align*}
so the coefficient in front of $\xi$ is of order $d^{\ee}\sqrt{h} =\Omega(d^{\delta/2})$ by the assumption $h\geq d^{-\frac1 2 + 3\delta}$.
Thus, the trigonometric terms precisely fit into the setting of Lemma~\ref{lem:trig}, and we will   apply Lemma~\ref{lem:trig} to estimate these terms.

Now let us estimate the terms of order one and two.
\begin{itemize}
    \item \emph{First- and lower-order terms.} We have
    \begin{align*}
        (\text{$\leq 1$st order})=1+ \E \aaa - \E \bbb -\E\ccc +\E \ddd\,.
    \end{align*}
    By Lemma~\ref{lem:trig}, we know $\E\aaa = O(d^{-1-2\ee}) = o(d^{-6\eta})$.
    For $\E \bbb$, we have
\begin{align*}
    -\E\bbb   = - \frac{h}{32d^{2\ee}} +\frac{h}{32d^{2\ee}}\E \cos(2d^\ee y_1) = - \frac{h}{32d^{2\ee}} + o(d^{-6\ee}),
\end{align*}
where we use Lemma~\ref{lem:trig} again. For $\E\ccc$, we have
\begin{align*}
 -\E \ccc = -\E \frac{2\hho x_1  \sin(d^\ee y_1)}{4d^\ee} 
 = \widetilde O(d^{-(1+\ee)}h)
 = o(d^{-5\eta}),
\end{align*}
where the last line is due to Lemmas~\ref{lem:trig} and \ref{lem:supcontrol}.  For $\E\ddd$, we have
\begin{align*}
    \E\ddd&= \E \frac{\sqrt{2\hht} \xi \sin(d^\ee y_1)}{4d^\ee} 
    = O(d^{-(1+\ee )}\sqrt{h})
    = o(d^{-5\ee}),
\end{align*}
where we use Lemma~\ref{lem:trig}.
Collecting together the terms, we have
\begin{align}\label{exp:1storder}
     (\text{$\leq 1$st order})=1-\frac{h}{32d^{2\ee}} + o(d^{-5\ee}).
\end{align}
\item \emph{Second-order terms.} For the reader's convenience, we have organized the terms which appear in the second-order Taylor expansion as Table~\ref{tab:second_order_terms}.

\begin{table}[h]
    \centering
    \begin{tabular}{c|cccc}
    & $O(d^{-2\eta})$ & $o(d^{-(3+1/3)\eta})$ & $o(d^{-(2+1/3) \eta})$ & $o_{\msf p}(d^{-(1+2/3)\eta})$ \\
    \hline \\
    $O(d^{-2\eta})$ & \eqref{eq:tab_term_1} & $o(d^{-4\eta})$ & $o(d^{-4\eta})$ & \eqref{eq:tab_term_2} \\
    $o(d^{-(3+1/3)\eta})$ & & $o(d^{-4\eta})$ & $o(d^{-4\eta})$ & $o_{\msf p}(d^{-4\eta})$ \\
    $o(d^{-(2+1/3) \eta})$ & & & $o(d^{-4\eta})$ & $o_{\msf p}(d^{-4\eta})$ \\
    $o_{\msf p}(d^{-(1+2/3)\eta})$ & & & & \eqref{eq:tab_term_3}
    \end{tabular}
    \caption{Terms which appear in the second-order Taylor expansion. The rows and columns are indexed by the terms $\aaa$, $\bbb$, $\ccc$, $\ddd$; refer to~\eqref{eq:order_of_terms}.}
    \label{tab:second_order_terms}
\end{table}
We now estimate the terms which are not covered by the table.
    Let us estimate the remaining terms one by one. First, by Lemma~\ref{lem:trig},
\begin{align}\label{eq:tab_term_1}
  \frac{1}{2}\E[\aaa^2] &=  \E\frac{\cos^2(d^\ee y_1)}{8d^{4\ee}} = \frac{1}{16d^{4\ee}} +  \E \frac{\cos(2d^\ee y_1)}{16d^{4\ee}} =\frac{1}{16d^{4\ee}} + o(d^{-8\eta})\,. 
\end{align} 
Next, by Lemma~\ref{lem:trig},
\begin{align}\label{eq:tab_term_2}
    \E[\aaa\ddd] &=\E\bigl[\frac{\sqrt{2\hht} \xi}{8 d^{3\ee}} \cos(d^\ee y_1)\sin(d^\ee y_1)\bigr] = \frac{\sqrt{2\hht}}{16 d^{3\ee}} \E[\xi \sin(2d^\ee y_1)] = o(d^{-7\ee}).
\end{align} 
Lastly, invoking Lemma~\ref{lem:trig} yet again,
 \begin{align}\label{eq:tab_term_3}
    \frac{1}{2}\E[\ddd^2] &= \E \frac{\hht \xi^2 \sin^2(d^\ee y_1)}{16d^{2\ee}}
    = \E\frac{\hht\xi^2}{32d^{2\ee}} - \E \frac{\hht \xi^2 \cos(2d^\ee y_1)}{32d^{2\ee}}  = \frac{\hht}{32 d^{2\ee}} + o(d^{-6\eta}).
\end{align}

Combining all together, we obtain,
  \begin{align} \label{exp:2ndorder}
        (\text{$2$nd order})&=  \frac{1}{16d^{4\ee}} + \frac{\hht}{32 d^{2\ee}} +o(d^{-4\ee})\,.
    \end{align}
\end{itemize}
Therefore, we combine \eqref{exp:1storder} and \eqref{exp:2ndorder} to conclude  
\begin{align*}
    \eqref{exp:1d_2} &\leq \exp\bigl[\frac{1}{16}d^{-4\ee}-\frac{h}{32d^{2\ee}}  + \frac{\hht}{32 d^{2\ee}} +o(d^{-4\ee}) \bigr]\\
    &=\exp\bigl[\frac{1}{16}d^{-4\ee}  +o(d^{-4\ee}) \bigr]\,,
\end{align*}
where the last line follows from the fact $\hht -h = \frac{(1-h)^2}{1+h^2}\, h -h \leq 0$. This implies \eqref{ineq:goal}, and hence the proof is complete. \end{proof}

\subsection{Upper bound on the spectral gap}\label{scn:spectral_gap_upper_bd}

Note that when $\ee < 1/4$, the adversarial potential defined in~\eqref{def:pi_bad} satisfies the assumptions of the following theorem, as a consequence of our computation in Lemma~\ref{lem:pibad}.

\begin{theorem}\label{thm:upper_bd_gap_separable}
    Consider a potential $V : \R^d \to\R$ which is separable: $V(\bx) = \sum_{i=1}^d v(x_i)$ for a function $v : \R\to\R$. Assume that:
    \begin{itemize}
        \item $V$ is symmetric about the origin, and $V(\bs 0) = \min V$.
        \item $V$ is $O(1)$-smooth.
        \item For the distribution $\pi_1 \propto \exp(-v)$, we have $\E_{x\sim \pi_1}[x^2] \asymp 1$.
    \end{itemize}
    Then, spectral gap of MALA with target distribution $\pi \propto \exp(-V)$ and step size $h\le 1$ satisfies
    \begin{align*}
        \lambda
        &\lesssim h\,.
    \end{align*}
\end{theorem}
\begin{proof}
    Consider the function $f : \R^d\to\R$ given by $f(\bx) := x_1$. Since $V$ is symmetric about the origin, we have $\E_\pi f = 0$.

    From the definition the spectral gap~\eqref{eq:lambda},
    \begin{align*}
        \lambda
        &\le \frac{\E_\pi[f \, ({\id} - T) f]}{\E_\pi[f^2]}
        \lesssim \E\displaylimits_{\substack{\bx \sim \pi \\ \by \sim T(\bx,\cdot)}}[{(x_1 - y_1)}^2]\,.
    \end{align*}
    Next, using the definition of the MALA kernel $T$, if $\xi$ is a standard Gaussian random variable, then
    \begin{align*}
        \E\displaylimits_{\substack{\bx \sim \pi \\ \by \sim T(\bx,\cdot)}}[{(x_1 - y_1)}^2]
        &= \E\displaylimits_{\substack{\bx \sim \pi \\ \by \sim Q(\bx,\cdot)}}[{(x_1 - y_1)}^2 \one_{\text{proposal}~\bx \to \by~\text{is accepted}}] \\
        &\le \E\displaylimits_{\substack{\bx \sim \pi \\ \by \sim Q(\bx,\cdot)}}[{(x_1 - y_1)}^2]
        = \E\displaylimits_{\bx \sim \pi}[{\{hv'(x_1) - \sqrt{2h} \xi\}}^2] \\
        &\le 2h^2 \E\displaylimits_{\bx \sim \pi}[{v'(x_1)}^2] + 4h \E[\xi^2]
        \lesssim h^2 \E\displaylimits_{\bx \sim \pi}[x_1^2] + h
        \lesssim h\,,
    \end{align*}
    by our assumptions. This completes the proof.
\end{proof}

\subsection{Auxiliary lemmas}

\begin{lemma}\label{lem:kl_of_adversarial}
    Let $\gamma := \Nor(0, I_d)$ and let $\pi$ be the adversarial target distribution defined in~\eqref{def:pi_bad}.
    Then,
    \begin{align*}
        \KL(\gamma \mmid \pi)
        &\le O(d^{1-4\ee}).
    \end{align*}
\end{lemma}
\begin{proof}
    From the definition of the KL divergence, if $\xi_1,\dotsc,\xi_d$ are i.i.d.\ random variables drawn according to $\gamma$, then
    \begin{align*}
        \KL(\gamma \mmid \pi)
        &= \int \gamma(\bx) \ln\bigl(\frac{Z^d}{{(2\uppi)}^{d/2}} \, \exp \ft(\bx)\bigr) \, \D \bx
        = d \ln \frac{Z}{\sqrt{2\uppi}} - \frac{1}{2d^{2\eta}} \sum_{i=1}^d \E \cos(d^\eta \xi_i).
    \end{align*}
    From our estimate of the normalizing constant in Lemma~\ref{lem:pibad},
    \begin{align*}
        d\ln \frac{Z}{\sqrt{2\uppi}}
        &= d\ln\bigl(1 + O(d^{-4\ee})\bigr)
        = O(d^{1-4\ee}).
    \end{align*}
    On the other hand, from the proof of Lemma~\ref{lem:trig},
    \begin{align*}
        - \frac{1}{2d^{2\eta}} \sum_{i=1}^d \E \cos(d^\eta \xi_i)
        = o(d^{1-4\eta}).
    \end{align*}
    The result follows.
\end{proof}

\section{Calculations for a Gaussian target distribution}\label{scn:gaussian}

In this section, we provide calculations for MALA when the target distribution $\pi$ is the standard Gaussian. 
Since MALA applied to the Gaussian distribution has a scaling limit in the sense of \cite{roberts1998optimal}, one would expect the mixing time of the Gaussian distribution to be of order $d^{1/3}$, and that is indeed what we show below.

\subsection{Upper bound}
First, we show that, under a warm start, the mixing time of MALA applied to the standard Gaussian mixes at $O(d^{1/3})$ rate.
\begin{proposition}\label{prop:gaussian_upper}
Let $\varepsilon > 0$, and let the target distribution $\pi$ be the standard Gaussian on $\R^d$. For a step size $h = cd^{-1/3}$, where $c > 0$ is a small constant, and an initial distribution $\mu_0$ that is $M_0$-warm with respect to $\pi$ such that $\log \frac{M_0}{\varepsilon h} = O(d^{1/3})$, the mixing time of MALA satisfies
\begin{align*}
    \tau_{\rm mix}(\varepsilon, \mu_0; {\rm TV})
        &\lesssim d^{1/3} \log\Bigl(\frac{M_0}{\varepsilon}\Bigr) \,.
\end{align*}
\end{proposition}

Using the results of Appendix~\ref{scn:other_distances}, the mixing time bounds can then be extended to the KL divergence, the chi-squared divergence, and the $2$-Wasserstein distance.

The proof crucially relies on the fact that when $h\approx d^{-1/3}$, the acceptance probability $A(\bx)$ (see~\eqref{eq:metro_adjust})
when $\bx \sim \pi$ is of order $\Omega(1)$ with high probability, which is formalized below.
\begin{lemma}\label{lem:gaussian_acc_prob}
    Let $\pi$ be the standard Gaussian. For $h = c_0 d^{-1/3}$, where $c_0 >0$ is sufficiently small, and $\bx \sim \pi$, there exists $c_1 > 0$ such that with probability at least $1 - 2\exp(-c_1 d^{1/3})$, it holds that $A(\bx) \ge 5/6$.
\end{lemma}
\begin{proof}[Proof of Proposition~\ref{prop:gaussian_upper}]
    We sketch the proof, following the $s$-conductance mixing time strategy outlined in Appendix~\ref{scn:upper_bd_overview}. Let $E := \{\bx \in\R^d \mid A(\bx) \ge 5/6\}$. Lemma~\ref{lem:gaussian_acc_prob} guarantees that $\pi(E) \ge 1 - 2\exp( - c_1 d^{1/3})$. By our assumption, we have $\log(\varepsilon h / M_0) = \Omega(d^{-1/3})$, so $\pi(E) \ge 1 - c' \sqrt{h} s$ for some constant $c'>0$, where $s := \varepsilon/(2M_0)$. Moreover, on the event $E$ we have (by Proposition~\ref{prop:basics}) that
    \begin{align*}
        \norm{T_{\bx} - Q_{\bx}}_{TV} = 1 - A(\bx) \le \frac 1 6 \,.
    \end{align*}
    Then the argument in the proof of Proposition~\ref{prop:s_cond} implies that the $s$-conductance, defined in \eqref{eq:C_s}, is lower bounded by $\msf C_s \gtrsim \sqrt h$, and Corollary~\ref{cor:mixing_from_s_conductance} gives the desired mixing time bound.
\end{proof}
\begin{proof}[Proof of Lemma~\ref{lem:gaussian_acc_prob}]
    Let $\bx \sim \pi$ and $\by \sim Q(\bx, \cdot)$.
    We will use $c$ to denote universal constants, which can change from line to line. First note that by concentration of the norm~\citep[Theorem 3.1.1]{vershynin2018highdimprob}, we have that for all $t > 0$, 
    \begin{align*}
        \Pr\bigl(\bigl\lvert\,\norm{\bx} - \sqrt d\,\bigr\rvert > t\bigr) \le 2 \exp(-ct^2)\,.
    \end{align*}
    As a result, the event 
    \begin{align*}
        E_1 := \bigl\{\bigl\lvert\,\norm{\bx} - \sqrt d\,\bigr\rvert \le t_1\bigr\}
    \end{align*}
    holds with probability at least $1- 2\exp( - ct_1^2)$.
    
    By the radial symmetry of the standard Gaussian, we can assume that the only non-zero coordinate of $\bx$ is the first coordinate: $\bx = (x_1, 0,\dotsc,0)$. Given $\bx$, we draw $\by$ by:
    \begin{align*}
        \by &= (1-h)\bx + \sqrt{2 h} \,\bxi, \qquad\ \bxi\sim \mc N(0, I_d)\,.
    \end{align*}
    We can write $\bxi = (\xi_1, \bxi_{-1})$, where $\xi_1 \sim \mc N(0, 1)$, and $\bxi_{-1}\sim \mc N(0, I_{d-1})$. By Gaussian concentration, the event
    \begin{align*}
        E_2 := \{|\xi_1| \le t_2\}
    \end{align*}
    holds with probability at least $1-2\exp(-ct_2^2)$, and the event
    \begin{align*}
        E_3 := \bigl\{\bigl\lvert\,\norm{\bxi_{-1}} - \sqrt d \,\bigr\rvert \le t_3\}
    \end{align*}
    hold with probability at least $1 - 2\exp(- ct_3^2)$. Define the quantities
    \begin{align*}
        \epsilon_1
        &:= \norm{\bx} - \sqrt d\,, \qquad \epsilon_2 := \xi_1\,, \qquad \epsilon_3 := \norm{\bxi_{-1}} - \sqrt d \,.
    \end{align*}
    Note that when $\pi$ is the standard Gaussian, a brief calculation using the definition~\eqref{eq:accept_prob} shows that $a(\bx,\by) = \exp(\frac{h}{4} \, (\norm{\bx}^2 - \norm{\by}^2))$. Then, on the event $E_1\cap E_2 \cap E_3$, we have that
    \begin{align*}
        \frac h 4\, \bigl\lvert\,\norm{\bx}^2 - \norm{\by}^2\,\bigr\rvert
        &= \frac h 4\, \lvert x_1^2 - [(1-h)x_1 + \sqrt{2h}\,\xi_1]^2 - 2h\, \norm{\bxi_{-1}}^2\rvert\\
        &= \frac h 4 \,\lvert (\sqrt d + \epsilon_1)^2 - [(1-h)\,(\sqrt d + \epsilon_1) + \sqrt{2h}\, \epsilon_2]^2 - 2h\, (\sqrt d + \epsilon_3)^2\rvert\\
        &= O(dh^3 + d^{1/2} h^2 t_1 + h^{3/2} d^{1/2} t_2 + d^{1/2} h^2 t_3)\,,
    \end{align*}
    assuming that $t_1 = O(d^{1/2})$.
    In fact, we take $t_1, t_3 = d^{1/6}$.
    If we take $t_2$ to be a sufficiently large constant (and the dimension $d$ is large), then we can ensure that the event $E_2 \cap E_3$ holds with probability at least $10/11$. With these choices,
    \begin{align*}
        \frac h 4\, \bigl\lvert\,\norm{\bx}^2 - \norm{\by}^2\,\bigr\rvert
        &= O(dh^3 + d^{2/3} h^2 + d^{1/2} h^{3/2})\,.
    \end{align*}
    Taking $h \le c/d^{1/3}$ for a sufficiently small constant $c > 0$, we can ensure that $a(\bx,\by) \ge 11/12$. Thus, on the event $E_1$, we have
    \begin{align*}
        A(\bx)
        &= \E[A(\bx, \by) \mid \bx]
        \ge \E[A(\bx,\by) \one_{E_2 \cap E_2} \mid \bx]
        \ge \frac{11}{12} \cdot \frac{10}{11}
        = \frac{5}{6}\,.
    \end{align*}
    This completes the proof.
\end{proof}

\subsection{Lower bound}

We show that when the step size is chosen as $h \gg d^{-1/3}$, then the conductance of the MALA chain with Gaussian target is exponentially small.

\begin{proposition}
    For every $\rr < 1/3$, if we take step size $h = d^{-\rr}$, then the conductance of the MALA chain is exponentially small:
    \begin{align*}
        \exists \delta > 0 \qquad\text{such that}\qquad \msf C
        &\lesssim \exp[-\Omega(d^\delta)]\,.
    \end{align*}
\end{proposition}
\begin{proof}
\renewcommand{\thefootnote}{\fnsymbol{footnote}}
We want to upper bound the conductance, defined in \eqref{eq:C}. It suffices to show that there exists an event $\eve \subseteq \R^d$ with $\pi(\eve) \ge 1/2$ such that
\begin{align*}
    \sup_{\bx \in \eve} \int Q(\bx,\by) A(\bx,\by) \, \D \by = \exp[-\Omega(d^\delta)]\,,
\end{align*}
see Proposition~\ref{prop:upper_conductance}.
Specifically, we will take $E := \{\bx\in\R^d \mid \norm{\bx} \leq \sqrt d\}$; note that
\begin{align*}
    \pi(E) = \frac{\Gamma(\frac d 2, 0) - \Gamma(\frac d 2, \frac d 2)}{\Gamma(\frac d 2)} > \frac 1 2\,.
\end{align*}
From the definition~\eqref{eq:accept_prob}, we have $A(\bx, \by) = a(\bx, \by)\wedge 1 \leq \sqrt{a(\bx, \by)}$.\footnote{One can check that the simple bound $A(\bx,\by) \le a(\bx,\by)$ is not enough for the proof to go through. A similar argument to upper bound the acceptance probability is made in \cite{hairer2014spectral}.} Since $V(\bx) = \frac 1 2 \norm{\bx}^2$, a little algebra using the definition~\eqref{eq:accept_prob} shows that
\begin{align*}
    a(\bx, \by) = \exp\bigl(\frac h 4 \, (\norm{\bx}^2 - \norm{\by}^2)\bigr)\,.
\end{align*}
Further calculations show that
\begin{align*}
    \int_{\R^d} Q(\bx, \by) A(\bx, \by)\, \D \by
    &\leq \int_{\R^d} Q(\bx, \by) {a(\bx, \by)}^{1/2} \, \D \by\\
    &= \int_{\R^d} \frac{1}{{(4\uppi h)}^{d/2}} \exp\bigl(- \frac{1}{4h}\,\norm{\by - (1-h)\bx}^2\bigr) \exp\bigl(\frac h 2 \, (\norm{\bx}^2 - \norm{\by}^2)\bigr) \, \D \by \\
    &= \frac{1}{{(4\uppi h)}^{d/2}} \int_{\R^d} \exp\Bigl( - \frac{1+h^2/2}{4h} \, \bigl\lVert \by - \frac{1-h}{1 + h^2/2}\, \bx\bigr\rVert^2 \Bigr) \, \D \by \\
    & \ \ \ \ \ \ \ \ \ \ \ \ \qquad\qquad\qquad\qquad\qquad\qquad\qquad \times{} \exp\bigl(\frac{h^2 \, (1 - h/4)}{1 + h^2/2} \, \norm{\bx}^2\bigr) \\
    &= \exp\Bigl(\frac{h^2 \, (1 - h/4)}{4 \, (1 + h^2/2)} \, \norm{\bx}^2 - \frac d 2 \ln\bigl(1 + \frac{h^2}{2}\bigr)\Bigr).
\end{align*}
For $\bx \in \eve$, we can bound this via
\begin{align*}
    \int_{\R^d} Q(\bx, \by) A(\bx, \by)\, \D \by
    &\le \exp\Bigl(\frac{h^2 \, (1 - h/4) d}{4 \, (1 + h^2/2)} - \frac d 2 \ln\bigl(1 + \frac{h^2}{2}\bigr)\Bigr)
    = \exp\Bigl( - \frac{h^3 d}{16} \, \bigl(1 + O(h)\bigr)\Bigr)
\end{align*}
which completes the proof.
\end{proof}

The next result shows that the spectral gap of the MALA chain is always upper bounded by the step size. Together with the preceding result, it implies that the mixing time of the MALA chain with Gaussian target is no better than $O(d^{1/3})$.

\begin{proposition}
    The spectral gap of MALA with Gaussian target distribution and step size $h$ satisfies
    \begin{align*}
        \lambda \lesssim h \,.
    \end{align*}
\end{proposition}
\begin{proof}
    This is a special case of Theorem~\ref{thm:upper_bd_gap_separable}.
\end{proof}

\bibliography{ref.bib}

\end{document}